\documentclass[a4paper, 10pt]{article}
\usepackage{srcltx,graphicx,epstopdf}
\usepackage{amsmath, amssymb, amsthm}
\usepackage{multirow}
\usepackage{diagbox}
\usepackage{color}
\usepackage{xcolor}
\usepackage{multirow}
\usepackage{subfigure} 
\usepackage{hyperref}
\usepackage[margin=1in]{geometry}
\usepackage[T1]{fontenc}
\usepackage{cleveref}
\crefname{section}{Section}{Sections}
\crefname{subsection}{Subsection}{Subsections}
\crefname{appendix}{Appendix}{Appendix}
\crefname{figure}{Figure}{Figures}
\crefname{table}{Table}{Tables}
\crefname{property}{Property}{Properties}
\crefname{theorem}{Theorem}{Theorem}
\usepackage{booktabs}
\usepackage{tabularx}
\usepackage{threeparttable}
\usepackage{enumerate}
\usepackage[normalem]{ulem}
\usepackage{caption}
\usepackage{subcaption}
\usepackage[notref,notcite]{showkeys}
\graphicspath{{images/}}

\newtheorem{theorem}{Theorem}[section]

\newtheorem{definition}{Definition}[section]

\newtheorem{remark}{Remark}[section]

\newtheorem{assumption}{Assumption}

\numberwithin{equation}{section}

\definecolor{electricpurple}{rgb}{0.75,0.0,1.0}
\definecolor{darkred}{rgb}{0.65,0,0}
\definecolor{green}{rgb}{0.0, 0.5, 0.0}

\newcommand\deletei{\bgroup\markoverwith{\textcolor{darkred}{\rule[0.5ex]{1pt}{1pt}}}\ULon}
\newcommand\deleteii{\bgroup\markoverwith{\textcolor{blue}{\rule[0.5ex]{2pt}{2pt}}}\ULon}

\title {Weak formulation and spectral approximation of a Fokker-Planck equation for neural ensembles} 

\author{Ling Yan\thanks{Beijing Computational Science Research Center, Beijing, China, 100193, email: {\tt yanling@csrc.ac.cn}}, 
~~Pei Zhang\thanks{Beijing Computational Science Research Center, Beijing, China, 100193, email: {\tt zhangpei@csrc.ac.cn}}, 
~~Yanli Wang\thanks{Beijing Computational Science Research Center, Beijing, China, 100193, email: {\tt ylwang@csrc.ac.cn}},
~~Zhennan Zhou\thanks{Institute for Theoretical Sciences, Westlake University, Hangzhou, China, 310030, email: {\tt zhouzhennan@westlake.edu.cn}}.}

\begin{document}
\maketitle
\begin{abstract}
    In this paper, we focus on efficiently and flexibly simulating the Fokker-Planck equation associated with the Nonlinear Noisy Leaky Integrate-and-Fire (NNLIF) model, which reflects the dynamic behavior of neuron networks. We apply the Galerkin spectral method to discretize the spatial domain by constructing a variational formulation that satisfies complex boundary conditions. Moreover, the boundary conditions in the variational formulation include only zeroth-order terms, with first-order conditions being naturally incorporated. This allows the numerical scheme to be further extended to an excitatory-inhibitory population model with synaptic delays and refractory states. Additionally, we establish the consistency of the numerical scheme. Experimental results, including accuracy tests, blow-up events, and periodic oscillations, validate the properties of our proposed method.

\end{abstract}

\vspace*{4mm}


\section{Introduction}\label{sec:introduction}
With the development of neuroscience, the analysis of large-scale neuron network models has attracted significant attention. In recent years, an increasing number of researchers have begun to employ mathematical theories and tools to analyze the activity of simulated neurons. Various stochastic methods \cite{newhall2010cascade, newhall2010dynamics} are frequently used to simulate these mathematical models. However, the computational complexity escalates as the number of neurons in the neuron networks increases. To address this issue, researchers developed the mean-field theory \cite{nykamp2000population, dumont2020mean, renart2004mean}, which has been the foundation of a growing number of models \cite{dumont2016noisy, dumont2024oscillations, nykamp2001population} that utilize its principles.
	
One of the most classic models in computational neuroscience is the Nonlinear Noisy Leaky Integrate-and-Fire (NNLIF) model, which utilizes partial differential equations to describe the dynamic activity of a single neuron's membrane potential $V(t)$ with statistical description, such as the distribution and the mean firing rate. The NNLIF model captures the temporal evolution of the membrane potential, considering both its decay and reset mechanisms. In the absence of external stimuli, the membrane potential decays exponentially towards the resting potential $V_L$ and we assume $V_L=0$ for the sake of simplicity in this paper. Provided that neurons receive excitatory synaptic inputs from other neurons, the voltage of the neuron continues to increase. Once the membrane potential reaches the threshold voltage $V_F$, the membrane potential will not continue to increase, but instead will instantaneously reduce to the reset voltage $V_R$. This process will continuously repeat as the neuron responds to the synaptic inputs. 

To characterize the dynamic behavior of the neuron population, the evolution equation for the probability density function, also known as the Fokker-Planck equation, has been derived as in \cite{nykamp2000population, brunel1999fast, caceres2011analysis}
	\begin{equation}
		\label{eq:1}
		\begin{cases}
			\partial _t p(v,t)+\partial _v\big[h(v,N(t))p(v,t)\big]-a(N(t))\partial _{vv}p(v,t)=0,\quad v\in (-\infty ,V_F]/\left \{V_R\right\}, \\
			p(v,0)=p^0(v),\quad p(-\infty,t)=p(V_F,t)=0,\\
			p(V^-_R,t)=p(V^+_R,t),\quad \partial _v p(V^-_R,t)=\partial _v p(V^+_R,t)+\frac{N(t)}{a},
		\end{cases}
	\end{equation}
where
	\begin{equation}
		\label{eq:1.1}
		\begin{cases}
			N(t)=-a(N(t))\partial_vp(V_F,t), \\
			h(v,N(t))=-v+bN(t),\\
			a(N(t))=a_0+a_1N(t).
		\end{cases}
	\end{equation}
Here, the density function $p(v,t)$ represents for a given time $t$ the probability of finding a neuron at voltage $v$. The connectivity parameter $b$ is positive for excitatory networks and negative for inhibitory networks. And $h(v,N(t))$ is the drift coefficient and $a(N(t))$ is the diffusion coefficient. 
$N(t)$ is the mean firing rate of the neurons at time $t$, and it determines the magnitude of the flux shift, which results in the complicated boundary condition at $v=V_R$. The density function $p(v,t)$ satisfies the mass conservation 
	\begin{equation}
		\label{eq:mass conservation}
		\int_{-\infty}^{V_F}p(v,t)\mathrm{d}v=\int_{-\infty}^{V_F}p^0(v)\mathrm{d}v=1.
	\end{equation}
 In addition, we suppose the initial value $p^0(v)$ satisfies the following assumptions \cite{carrillo2013classical} 
 \begin{assumption}[Assumption $\mathbf{(H)}$]
 \label{thm:H}
 $p^0(v)\in C^1\big((-\infty,V_R)\cup(V_R,V_F]\big) \cap L^1\big((-\infty,V_R)\cup(V_R,V_F)\big)$, is non-negative with $p^0(V_F)=0$ and 
 \begin{equation}
     \lim_{v\to-\infty}\partial_v p^0(v)=0.
 \end{equation}
 \end{assumption}

In recent years, there have been a number of significant developments about the NNLIF model, including theoretical analysis and numerical simulation. In \cite{caceres2011analysis, carrillo2013classical}, the author investigates analytical properties of the solutions, steady states and finite time blow-up phenomenon for excitatory-average networks or inhibitory-average networks. Building upon this, multiple population models have been studied in \cite{caceres2016blow, Cáceres2018Analysis} in various aspects, including stationary solutions and global solution behavior. Additionally, some researchers focus on the variants of the NNLIF model, such as models with the refractory state, and synaptic delay \cite{Cáceres2018Analysis, caceres2019global}. Besides, there has been a growing interest in more complex solution behavior such that the emergence of cascade assemblies \cite{newhall2010cascade, newhall2010dynamics} and periodic oscillations \cite{brunel1999fast, brunel2000dynamics}, and the generalized solution of the NNLIF model has been proposed to allow blow-up events by introducing the dilated time scale \cite{dou2022dilating}, which may serve as a foundation tool for analyzing these complex dynamics in the PDE models.

Due to the non-linear drift and diffusion terms and the unique flux shift mechanism, it is challenging to investigate the Fokker-Planck equation \eqref{eq:1} with analytical techniques. Therefore, numerical methods are often desired to approximate the exact solution of the Fokker-Planck equation \eqref{eq:1}. More specifically, we seek an efficient and flexible numerical method that can not only handle the complex boundary conditions and nonlinear terms but is also expected to be applicable to various forms of the Fokker-Planck equations for neuron networks, including an excitatory-inhibitory populations model with synaptic delays and refractory states.

Various numerical methods have been proposed to simulate the NNLIF PDE model. In \cite{caceres2011numerical}, a finite difference method, which combines the WENO construction and the Chang-Cooper method, is applied to this model with the explicit third-order Runge-Kutta method. In \cite{sharma2019numerical}, a comparison was made between a finite element method using Euler’s backward difference scheme and the WENO finite difference method, with the FEM method showing potential advantages for nonlinear noisy LIF models. Additionally, a discontinuous Galerkin method for the variant model with delay and refractory periods is proposed in \cite{sharma2020discontinuous} with justified stability properties.

In recent years, developing structure-preserving methods for this model has appeared to be a new trend. A conservative and conditionally positivity-preserving scheme is proposed in \cite{hu2021structure} in the framework of the finite difference method, and the semi-discrete scheme is also proved to satisfy the discrete relative entropy estimation. Based on \cite{hu2021structure}, the authors propose an unconditionally positivity-preserving and asymptotic preserving numerical scheme in \cite{he2022structure} and a specialized multi-scale solver, integrating macroscopic and microscopic solvers, is presented in \cite{Du2024Synchronization} along with a rigorous numerical analysis for hybridizing both solvers.

To further reduce the computational complexity, there is also a trend to reduce the complex PDE dynamics to the hierarchy of moment equations, which naturally implies a computational tool. A moment method based on the maximum entropy principle is introduced in \cite{zhang2019coarse} to approximate the master equation of interacting neurons. In \cite{zhang2024spectral}, an asymptotic preserving numerical method in the framework of the spectral method without the above limitation is introduced, where the construction of the basis function can satisfy the dynamic boundary conditions exactly, and the spectral convergence is validated. However, unnecessary boundary conditions are imposed in its variational formulation, and the extension of the method in \cite{zhang2024spectral} in the current form is rather limited.



In this paper, we focus on developing a flexible spectral-Galerkin method for solving the Fokker-Planck equation \eqref{eq:1} in a semi-unbounded region. Spectral methods have demonstrated significant advantages for simulating the PDE model as shown in \cite{zhang2024spectral}, particularly in terms of spectral accuracy. That method in \cite{zhang2024spectral} primarily utilizes the Legendre spectral-Galerkin method to solve a specific variational formulation of the Fokker-Planck equation. In this variational formulation, three types of boundary conditions are imposed,  including the Dirichlet boundary conditions, the continuity condition, and the dynamical derivative conditions. However, due to the strict constraints imposed by the derivative conditions, it is challenging to extend such a method to models with more complex derivative boundary conditions, such as population models with synaptic delays and refractory periods. To overcome these limitations, we propose a Laguerre-Legendre spectral-Galerkin method (LLSGM) based on an alternative variational form. We reconstruct the weak formulation to only include the essential boundary conditions and the dynamical derivative boundary conditions are naturally implied. Consequently, the proposed numerical method can be readily extended to models with variant derivative boundary conditions. Additionally, we demonstrate the consistency of the numerical scheme with the Fokker-Planck equation, ensuring the reliability of the computational results.

We also carry out extensive numerical experiments with the proposed numerical method. We test the convergence in the temporal and spatial directions and the ability to capture the blow-up events in finite time. Additionally, we extend the proposed method to the excitatory-inhibitory populations with synaptic delays and refractory states and observe the periodic oscillations. The results indicate that the proposed method can be effectively applied to the variants of the Fokker-Planck equation \eqref{eq:1}. We stress that, in contrast to the spectral method in \cite{zhang2024spectral}, the variational formulation presented here includes only zeroth-order boundary conditions, while the derivative boundary conditions in \eqref{eq:1} are implicitly contained within the weak form, ensuring the general applicability of the method.

The organization of the paper is as follows. In Sec. \ref{sec:weak_form}, we introduce the weak solution and classical solution of the Fokker-Planck equation \eqref{eq:1}, and develop their conditional equivalence. In Sec. \ref{sec:scheme}, we construct the spectral approximation based on the weak solution, where we introduce the basis functions to discrete the spatial direction and the semi-implicit scheme employed for discretizing the temporal direction. Furthermore, the consistency of the numerical scheme is verified. In Sec. \ref{sec:two populations}, we extend our method to the two-populations model with synaptic delay and refractory period. In Sec. \ref{sec:numerical_test}, we present numerical numerical experiments, including the one-population model and the two-population models.

\section{Weak solution}\label{sec:weak_form}

In this section, we present a detailed description of both the weak solution and the classical solution to the PDE model \eqref{eq:1}-\eqref{eq:1.1}. Additionally, we demonstrate the equivalence between the weak solution and the classical solution under certain conditions, establishing a critical theoretical foundation for the proposed numerical method in later sections.

First of all, we consider the definition of classical solution as presented in \cite{liu2022rigorous}.

\begin{definition}[classical solution]\label{classical solution}
    $(p(v,t),N(t))$ is a classical solution of the system \eqref{eq:1}-\eqref{eq:1.1} in the time interval $[0,T]$ for any given $0<T<+\infty$ with Assmps. $\mathbf{(H)}$ \ref{thm:H} under the following conditions:
    \begin{itemize}
        \item[(1)] $N(t)$ is a continuous function for $t\in [0,T]$,
        \item[(2)] $p(v,t)$ is continuous in the region $\{(v,t):-\infty<v\le V_F, t\in[0,T]\}$,
        \item[(3)] $\partial_{vv}p(v,t)$ and $\partial_tp(v,t)$ are continuous in the region $\big\{(v,t): v\in (-\infty,V_R)\cup(V_R,V_F], t \in (0,T]\big\}$,
        \item[(4)] $\partial_vp(V_R^-,t)$, $\partial_vp(V_R^+,t)$ and $\partial_vp(V_F^-,t)$ are well defined for $t\in(0,T]$,
        \item[(5)] For $t\in (0,T]$, $\partial_vp(v,t)\to0$ where $v\to-\infty$,
        \item[(6)] $p(v,t)$ satisfies \eqref{eq:1}-\eqref{eq:1.1} and $p(v,0)=p^0(v)$ for $v\in (-\infty,V_R)\cup(V_R,V_F] $.
    \end{itemize}
\end{definition}
\noindent Here, the fifth condition involves the limit of the derivative at negative infinity. For simplicity, we will use the notation $ \partial_v p(-\infty,t)=0$ to denote 
\begin{equation}
    \lim_{v\to -\infty}\partial_v p(v,t)=0.
\end{equation}
Based on the classical solution $p(v,t)$ as defined in Def. \ref{classical solution}, we derive a weak solution by multiplying both sides of \eqref{eq:1} with a test function $\phi(v)\in C_0^{\infty}\big((-\infty,V_F]\big)$, where $v\partial_v\phi(v)$, $\partial_v \phi(v)$ and $v\phi(v)\in L^{\infty}\big((-\infty,V_F]\big)$. We then perform integration by parts to obtain
\begin{equation}
\label{eq:3}
    \begin{split}
	\int^{V_F}_{-\infty} \Big[\partial _t p(v,t)\phi(v)-h(v,N(t))p(v,t)\partial_v\phi(v)+a(N(t))\partial_vp(v,t)\partial_v\phi(v)\Big]\mathrm{d}v\\
	+\Big(h(v,N(t))p(v,t)\phi(v)\big|_{-\infty}^{V_R^-}+h(v,N(t))p(v,t)\phi(v)\big|_{V_R^+}^{V_F}\Big)\\
        -\Big(a(N(t))\partial_vp(v,t)\phi(v)\big|_{-\infty}^{V_R^-}+a(N(t))\partial_vp(v,t)\phi(v)\big|_{V_R^+}^{V_F}\Big)=0.
    \end{split}
\end{equation}
Defining the operator $ \mathcal{B}(p,\phi)$ as 
\begin{equation}
    \label{eq:integration_weak_sol}
    \mathcal{B}(p,\phi)=\int^{V_F}_{-\infty} \Big[\partial _t p(v,t)\phi(v)-h(v,N(t))p(v,t)\partial_v\phi(v)+a(N(t))\partial_vp(v,t)\partial_v\phi(v)\Big]\mathrm{d}v,
\end{equation}
then by applying the boundary conditions \eqref{eq:1.1}, \eqref{eq:3} reduces to
\begin{equation}
\label{eq:4}
    \begin{split}
        \mathcal{B}(p,\phi)
        +a(N(t))\partial _v p(V_F,t)\big[\phi(V_R)-\phi(V_F)\big]=0,
    \end{split}
\end{equation}
where $p(v,t)$ satisfies
\begin{equation}
\label{eq:6}
        p(-\infty,t)=0,\qquad p(V_F,t)=0,\qquad p(V^-_R,t)=p(V^+_R,t).
\end{equation}
Thus, the weak solutions are introduced formally as below. 

\begin{definition}[weak solution] \label{weak_solution}
    Define the trial function space 
    \begin{equation}
    \label{eq:trial_function_space}
        U(\Omega)=\Big\{p(v)\in \mathbb{H}^1(\Omega)~\big|~ p(V_F)=p(-\infty)=0,\, p(V_R^-)=p(V_R^+)\Big\},
    \end{equation}
    where $\Omega=(-\infty,V_R)\cup(V_R,V_F)$.
    Let $p(v,t)\in C^1\big([0,T];U(\Omega)\big)$, $N(t)\in C([0,T])$ and $p(v,0)=p^0(v)$ satisfying Assmps. $\mathbf{(H)}$ \ref{thm:H} for $v\in \Omega$. Additionally, assuming that $p(v,t)$ is differentiable at $V_F$, $V_R^+$ and $V_R^-$. 
    
    Then $(p(v,t),N(t))$ is a weak solution of \eqref{eq:1}-\eqref{eq:1.1} if for any test function $\phi(v)\in \mathbb{H}^1\big((-\infty,V_F]\big)$ with $v\partial_v\phi\in\mathbb{H}^1\big((-\infty,V_F]\big)$, $\phi(v)$ and $v\phi(v)\in L^{\infty}\big((-\infty,V_F]\big)$, the following equation
\begin{equation}
\label{eq:7}
\begin{split}
    \mathcal{B}(p,\phi)+a(N(t))\partial _v p(V_F,t)\big[\phi(V_R)-\phi(V_F)\big]=0
\end{split}
\end{equation}
is satisfied for all $t\in(0,T]$.
\end{definition} 
\noindent Here, $C^1\big([0,T];U(\Omega)\big)$ represents the space comprising all continuous functions $f(v,t)$ from $[0,T]$ to $U(\Omega)$ with 
\begin{equation}
\label{eq:7.1}
    \left \| f(t) \right \| _{C([0,T];X)} :=\max _{0 \leq t \leq T}\left \| f(t) \right  \|_X < \infty,
\end{equation}
and first derivatives are also continuous with
\begin{equation}
\label{eq:7.2}
    \left \| \partial_t f(t) \right \| _{C([0,T];X)} :=\max_{0\le t\le T}\left \| \partial_t f(t) \right  \|_X <\infty,
\end{equation}
where $X=U(\Omega)$. The set of bounded functions is denoted as $L^{\infty}\big((-\infty,V_F]\big)$. $\mathbb{H}^1\big((-\infty,V_F]\big)$ is used to denote the abbreviation of classical Sobolev space $W^{1,2}\big((-\infty,V_F]\big)$. It should be noted that the boundary conditions satisfied by the weak solution in Def. \ref{weak_solution} only include the zeroth-order boundary conditions, while the derivative boundary condition in \eqref{eq:1} is implied within the weak solution as a natural boundary condition. This treatment of boundary conditions simplifies the variational formulation of the problem \eqref{eq:1} by removing the need to enforce higher-order derivatives at the boundaries.

\begin{remark}
The model extension discussed in Sec. $\ref{sec:two populations}$ further supports the reasonableness of retaining only the zeroth-order boundary conditions as essential boundary conditions. This demonstrates that the numerical method based on the weak solution presented in Sec. $\ref{sec:scheme}$ is both practical and effective.
\end{remark}

Before constructing the numerical scheme, we first demonstrate the relationship between the weak solution in Def. \ref{weak_solution} and the classical solution in Def. \ref{classical solution}.

\begin{theorem}[Equivalence to the classical solution]
\label{Equivalence to the classical solution}
    Given $p(v,t)\in C^1\big([0,T];U(\Omega)\big)$, if $(p(v,t), N(t))$ is a classical solution of \eqref{eq:1}-\eqref{eq:1.1}, then $(p(v,t), N(t))$ is a weak solution. Conversely, if $(p(v,t), N(t))$ is a weak solution of \eqref{eq:1}-\eqref{eq:1.1} and $p(v,t)\in C^1\big((0,T];C^2((-\infty,V_R)\cup (V_R,V_F])\big)$, then $(p(v,t), N(t))$ is a classical solution of \eqref{eq:1}-\eqref{eq:1.1}.
\end{theorem}

\begin{proof}

    First, it is obvious that if $p(v,t)\in C^1\big([0,T];U(\Omega)\big)$ and $(p(v,t),N(t))$ is the classical solution of \eqref{eq:1}-\eqref{eq:1.1}, then $(p(v,t),N(t))$ is also a weak solution. 

    Then, given that $p(v,t) \in C^1\big((0,T];C^2((-\infty,V_R)\cup (V_R,V_F])\big)$, it is suffices to verify if the weak solution $(p(v,t),N(t))$ satisfies the conditions in Def. \ref{classical solution}. It is easy to derive the conditions (1) - (4) in Def. \ref{classical solution} hold from Def. \ref{weak_solution}. Hence, it only remains to prove the conditions (5) - (6) in Def. \ref{classical solution} are valid. 
    
    It follows from \eqref{eq:trial_function_space} that 
    \begin{equation} \label{eq:BC}
        \big[h(v,N(t))p(v,t)\phi(v)\big]\big|_{-\infty}^{V_R^-}+\big[h(v,N(t))p(v,t)\phi(v)\big]\big|_{V_R^+}^{V_F}=0.
    \end{equation}
    Then using integration by parts and \eqref{eq:BC}, \eqref{eq:7} is rewritten into
    \begin{equation}
    \label{eq:10}
    \begin{split}
        \int_{-\infty}^{V_{F}} \mathcal{F}[p(v,t)]\phi(v) \mathrm{d}v+a(N(t))\phi(V_R)\big[\partial_vp(V_R^-,t)-\partial_vp(V_R^+,t)+\partial_vp(V_F,t)\big]\\
        -a(N(t))\partial_vp({-\infty},t)\phi({-\infty})=0,
    \end{split}
\end{equation}
where 
\begin{equation}
        \label{eq:FP_left}
        \mathcal{F}[p(v,t)]=\partial _t p(v,t)+\partial _v\big[h(v,N(t))p(v,t)\big]-a(N(t))\partial _{vv}p(v,t).
    \end{equation}
Taking
\begin{equation}
    \label{eq:11}
    \phi(v) \in \mathbb{W}_1=\left\{\phi(v)\in \mathbb{W}(-\infty,V_F):\phi(V_R)=\phi(-\infty)=0\right\},
\end{equation}
where 
\begin{equation}
    \mathbb{W}(-\infty,V_F) = \left\{\phi(v)\in \mathbb{H}^1(-\infty,V_F]:v\partial_v\phi\in \mathbb{H}^1(-\infty,V_F]\ \text{and}\ \phi,v\phi \in L^{\infty}(-\infty,V_F]\right\},
\end{equation}
hence, \eqref{eq:10} is reduced to
\begin{equation}
    \label{eq:12}
    \int_{-\infty}^{V_{F}} \mathcal{F}[p(v,t)]\phi(v) \mathrm{d}v=0.
\end{equation}
Due to the continuity of $\mathcal{F}[p(v,t)]$ and the arbitrariness of $\phi(v)$, we obtain
\begin{equation}
    \label{eq:13}
    \mathcal{F}[p(v,t)]=0.
\end{equation}
Next, taking 
\begin{equation}
    \label{eq:W_2_phi}
    \phi(v) \in \mathbb{W}_2=\left\{\phi(v)\in \mathbb{W}(-\infty,V_F):\phi(V_R)=0\right\}
\end{equation}
and substituting \eqref{eq:13} into \eqref{eq:10}, the following identity holds
\begin{equation}
    \label{eq:14}
    a(N(t))\partial_vp(-\infty,t)\phi(-\infty)=0,
\end{equation}
which implies
\begin{equation}
    \label{eq:15}
    \partial_vp(-\infty,t)=0.
\end{equation}
Finally, applying \eqref{eq:13} and \eqref{eq:15} to \eqref{eq:10}, we can derive 
\begin{equation}
    \label{eq:16}
    a(N(t))\phi(V_R)\left[\partial_vp(V_R^-,t)-\partial_vp(V_R^+,t)+\partial_vp(V_F,t)\right]=0.
\end{equation}
Hence, due to the arbitrariness of $\phi(v)$, it holds 
\begin{equation}
    \label{eq:17}
    \begin{split}
        \partial_vp(V_R^-,t)&=\partial_vp(V_R^+,t)-\partial_vp(V_F,t)\\
        &=\partial_vp(V_R^+,t)+\frac{N(t)}{a(N(t))}.
    \end{split}
\end{equation}
Consequently, from \eqref{eq:13}, \eqref{eq:15} and \eqref{eq:17}, we can deduce 
conditions (5) - (6) in Def. \ref{classical solution} hold, and the proof is completed.
    
\end{proof}

\section{Numerical scheme }\label{sec:scheme}

In this section, we develop a numerical scheme for the Fokker-Planck equation using a spectral method combined with semi-implicit time discretization to obtain a fully discrete scheme. 
Additionally, we demonstrate the consistency of the numerical scheme with the Fokker-Planck equation in the nonlinear case.

\subsection{Weak formulation and basis function}\label{sec:one population}

In this section, we introduce the weak formulation based on \eqref{eq:4}, and a set of basis functions is constructed. Based on \eqref{eq:4}, the variational formulation of \eqref{eq:1} is as bellow 
\begin{equation}
\label{eq:variational formula}
\begin{cases}
    \text{Find $p(v,t)\in W$ such that}\\
    \begin{aligned}
        \mathcal{B}(p,\phi)+a(N(t))\partial _v p(V_F,t)\left[\phi(V_R)-\phi(V_F)\right]=0,\quad \forall \phi \in W,
    \end{aligned}
\end{cases}
\end{equation}
where $W \subseteq U$ is the trial function space, with $U$ defined in \eqref{eq:trial_function_space}. Since we employ the Galerkin spectral method for spatial discretization, the test function space coincides with the trial function space $W$.


The trial function space $W$ is spanned by a set of basis functions, and one must carefully select function basis to handle the complex boundary conditions \eqref{eq:6} of the problem \eqref{eq:1}. Denoting the set of basis functions of the trial function space $W$ by $\left \{ \psi_k \right \} _{k=1}^{\infty }$, we can expand the approximate solution $p(v,t)\in W$ as
 \begin{equation}
     p(v,t)=\sum_{k=1}^{\infty}u_k(t)\psi_k(v).
 \end{equation}
To satisfy the continuity requirement of the density function $p(v,t)$ at $V_R$, we consider the use of a subset of basis functions designed to the boundary conditions \eqref{eq:6}, complemented by another subset designed to satisfy the homogeneous conditions. Specifically, let $W=W_1+W_2$, where any density function $p(v,t)\in W_1$ satisfies the boundary conditions \eqref{eq:6} and any $p(v,t)\in W_2$ satisfies the following homogeneous boundary conditions 
\begin{equation}
\label{eq:31}
			p(-\infty,t)=0, \qquad 
			p(V_F,t)=0, \qquad 
			p(V^-_R)=p(V^+_R)=0.
\end{equation}
This decomposition method enables us to handle complex boundary conditions more effectively while enhancing the flexibility of the numerical method. This flexibility primarily lies in the choice of basis functions in $W_1$, which address the boundary conditions, while $W_2$ is used to improve accuracy within the interval.

To keep the dimension of $W_1$ as low as possible, we observe that the four boundary values to consider from the boundary conditions \eqref{eq:6} are the values of the density function $p(v,t)$ at $V_F$, $V_R^+$, $V_R^-$, and negative infinity. However, there are only three conditions in the boundary conditions \eqref{eq:6}. Hence, only a single basis function $g_1(v)$ is needed. We can define $W_1$ as  
\begin{equation}
\label{eq:33}
		W_1=\mathrm{span}\left\{g_1(v)\right\},
\end{equation}
where the basis function should satisfy 
\begin{equation}
\label{eq:32}
		\begin{cases}
			g_1(-\infty)=0,\\
			g_1(V_R)=1,
		\end{cases}
		v\in(-\infty,V_R),
		\qquad
		\begin{cases}
			g_1(V_F)=0,\\
			g_1(V_R)=1,
		\end{cases}
		v\in(V_R,V_F].
\end{equation}
Specifically, we can determine a simple form for the basis function $g_1(v)$, which is a piecewise function composed of an exponentially decaying function and a linear function
\begin{equation}
\label{eq:34}
		g_1=
		\begin{cases}
			e^{-\beta(V_R-v)/2},\quad v\in(-\infty,V_R),\\
			\frac{v-V_F}{V_R-V_F},\quad v\in(V_R,V_F],
		\end{cases}
\end{equation}
where $\beta$ is an adjustable parameter. It should be noted that the basis function $g_1(v)$ can also be chosen in other forms.

Next, we proceed to the construction of $W_2$. Given the discontinuity of the interval $\Omega=(-\infty,V_R)\cup (V_R,V_F)$ at point $V_R$, handling homogeneous boundary conditions \eqref{eq:31} is non-trivial. To address this issue, we divide the interval $\Omega$ into two subintervals, namely $(-\infty,V_R)$ and $(V_R,V_F)$. Specifically, we define 
\begin{equation}
    W_2 \subseteq \hat{P}_{\infty}(-\infty,V_R)+P_{\infty}(V_R,V_F),
\end{equation}
where 
\begin{equation}
    \hat{P}_{\infty}(-\infty,V_R)=\left\{e^{-\frac{-v+V_R}{2}}f\big|f\in P_{\infty}(-\infty,V_R)\right\}
\end{equation}
and \( P_{\infty}(a,b) \) represents the set of all real polynomials over the interval \( (a,b) \). According to the homogeneous boundary conditions \eqref{eq:31}, the values of the density function $p(v,t)$ at the boundaries of the intervals \((-\infty, V_R)\) and \((V_R, V_F)\) should be zero. This implies that the basis functions must vanish at these boundary points. Thus, \( W_2 \) can be rewritten in the following form
\begin{equation}
\label{eq:35}
		W_2=\hat{X}_{(-\infty,V_R)}+X_{(V_R,V_F)},
\end{equation}
where 
\begin{equation}
\label{eq:36}
\begin{aligned}
\hat{X}_{(-\infty,V_R)}&=\big\{\varphi\in \hat{P}_{\infty}(-\infty,V_R):\varphi(-\infty)=\varphi(V_R)=0  \big\},\\
X_{(V_R,V_F)}&=\big\{\varphi\in P_{\infty}(V_R,V_F):\varphi(V_R)=\varphi(V_F)=0\big\}.
\end{aligned}
\end{equation}

Here, we consider the construction of basis functions in $\hat{X}_{(-\infty, V_R)}$ and $X_{(V_R, V_F)}$. Because $\hat{X}_{(-\infty, V_R)}$ is defined on a semi-infinite domain, we can utilize Laguerre functions to span $\hat{X}_{(-\infty, V_R)}$. The generalized Laguerre functions \cite{shen2011spectral} of degree $n$ are defined by
\begin{equation}
    \label{eq:laguerre function}
    \hat{\mathcal{\ell}}^{(\alpha)}_n(x)=e^{-x/2}\mathcal{\ell}^{(\alpha)}_n(x)=\frac{x^{-\alpha}e^{x/2}}{n!}\frac{\mathrm{d}^n}{\mathrm{d}x^n}(x^{n+\alpha}e^{-x}),\quad x\in \mathbb{R}_+,\quad \alpha>-1,\quad n\ge0,
\end{equation}
where $\mathcal{\ell} ^{(\alpha)}_n(x)$ is Laguerre polynomials.
These Laguerre functions are orthogonal with the weight function $\hat{\omega}_{\alpha}(x)=x^{\alpha}$, namely,
\begin{equation}
    \label{eq:laguerre orthogonal}
    \int_0^{+\infty}\hat{\mathcal{\ell}}^{(\alpha)}_n(x)\hat{\mathcal{\ell}}^{(\alpha)}_m(x)\hat{\omega}_{\alpha}(x)\mathrm{d}x=\gamma_n^{(\alpha)}\delta_{mn},
\end{equation}
where
\begin{equation}
    \label{eq:gamma}
    \gamma_n^{(\alpha)}=\frac{\Gamma(n+\alpha+1)}{n!}.
\end{equation}
The three-term recurrence relation of Laguerre functions is
\begin{align}
&\hat{\mathcal{\ell}}^{(\alpha)}_0(x)=e^{-x/2}, \quad \hat{\mathcal{\ell}}^{(\alpha)}_1(x)=(\alpha+1-x)e^{-x/2},\nonumber \\
&(n+1)\hat{\mathcal{\ell}}^{(\alpha)}_{(n+1)}(x)=(2n+\alpha+1-x)\hat{\mathcal{\ell}}^{(\alpha)}_n(x)-(n+\alpha)\hat{\mathcal{\ell}}^{(\alpha)}_{(n-1)}(x).\label{eq:three-term relation}
\end{align}
Since $\hat{\ell}^{(\alpha)}_n(x)$ is defined on $(0,+\infty)$, we construct the basis function $h_k^L(v)$ defined on $(-\infty,V_R)$ with a coordinate transformation
\begin{equation}
\label{eq:37}
		h_k^L(v)=\hat{\mathcal{\ell}}_k^{(0)}(-v+V_R)+a_1\hat{\mathcal{\ell}}_{k+1}^{(0)}(-v+V_R).
\end{equation}
where
\begin{equation}
\label{eq:La at 0}
    \hat{\mathcal{\ell}}_n^{(0)}(0)=1,\quad \lim_{v\to-\infty}\hat{\mathcal{\ell}}_n^{(0)}(v)=0.
\end{equation}
Because the solution satisfies $p(-\infty,t)=p(V_R^-,t)=0$, basis functions $h_k^L(v)$ should satisfy the following conditions
\begin{equation}
        h_k^L(-\infty)=0,\qquad
        h_k^L(V_R^-)=0.
\end{equation}
Therefore, we get $a_1=-1$.

For the bounded region on the right side of \( V_R \), we employ basis functions that differ from those used in the left interval \((-\infty, V_R)\). Specifically, Legendre polynomials \( L_k(v) \) are chosen to construct the basis functions on the interval \( (V_R, V_F) \). Since the domain of Legendre polynomials is \( [-1, 1] \), a coordinate transformation
\begin{equation}
    x(v)= \frac{v-\frac{V_F+V_R}{2}}{\frac{V_F-V_R}{2}},\quad v\in(V_R,V_F)
\end{equation}
is necessary.
Furthermore, the basis functions in the interval \( (V_R, V_F) \) are given by 
\begin{equation}
\label{eq:38}
    h_k^R(v)=\mathcal{H}_k(v)+b_1\mathcal{H}_{k+1}(v)+b_2\mathcal{H}_{k+2}(v),
 \end{equation}
 where
 \begin{equation}
     \mathcal{H}_k(v)= L_k(x).
 \end{equation}
Similarly, from \eqref{eq:35} and \eqref{eq:36}, we obtain 
\begin{equation}
 \label{eq:38.1}
        h_k^R(V_R^+)=0,\quad h_k^R(V_F)=0.
\end{equation} 
By substituting \eqref{eq:38.1} into \eqref{eq:38}, we get that \( b_1 = 0 \) and \( b_2 = -1 \).

In conclusion, we have designed a suitable set of basis functions for the trial function space $W$ as follows
 \begin{equation}
 \label{eq:39}
    \left\{\psi_k\right\}^{\infty}_{k=1}=\left\{g_1,h_0^L,h_1^L,\cdots,h_{M-1}^L,\cdots,h_0^R,h_1^R,\cdots,h_{M-1}^R,\cdots\right\}.
\end{equation}
Then the solution of \eqref{eq:variational formula} can be written as 
\begin{equation}
\label{eq:41}
		p(v,t)=\sum_{k=1}^{\infty}u_k(t)\psi_k(v).
\end{equation}

\subsection{Fully discrete scheme}
\label{sec:Fully_dis}
In this section, we introduce a Laguerre-Legendre spectral-Galerkin method (LLSGM) for the spatial discretization of the Fokker-Planck equation \eqref{eq:1}, transforming it into a system of ordinary differential equations (ODEs) in time. We then derive the fully discrete scheme for the weak formulation in two steps. First, we construct a finite-dimensional trial function space and apply the spectral method for spatial discretization. In the second step, a semi-implicit method is used for temporal discretization, resulting in a nonlinear fully discrete scheme.


To obtain the finite-dimensional trial function space \( W \), we denote the approximation space $W_M$ as
\begin{equation}
    W_M=W_1+W_2^M,
\end{equation}
where \( W_1 \) is a one-dimensional space and $W_2^M$ is the truncated space of $W_2$. From this, we derive the following form of $W_2^M$ using \eqref{eq:35} and \eqref{eq:36}
\begin{equation}
\label{eq:35.1}
		W_2^M=\hat{X}_{M(-\infty,V_R)}+X_{M(V_R,V_F)},
\end{equation}
where 
\begin{equation}
\label{eq:36.1}
\begin{aligned}
    \hat{X}_{M(-\infty,V_R)}&=\big\{\varphi\in \hat{P}_{M}(-\infty,V_R):\varphi(-\infty)=\varphi(V_R)=0  \big\},\\
    X_{M(V_R,V_F)}&=\big\{\varphi\in P_{M}(V_R,V_F):\varphi(V_R)=\varphi(V_F)=0\big\}.
\end{aligned}
\end{equation}
The notation $\hat{P}_{M}(-\infty,V_R)$ represents
\begin{equation}
    \hat{P}_{M}(-\infty,V_R)=\left\{e^{-\frac{-v+V_R}{2}}f~\big|~f\in P_{M}(-\infty,V_R)\right\}
\end{equation}
and $P_M(a,b)$ denotes the set of all real polynomials of degree at most $M$ in the interval $(a,b)$. In $W_1$, we have identified that a single basis function $g_1(v)$ satisfies the boundary conditions \eqref{eq:32}. Therefore, the dimension of the approximation space  $W_M$ is $2M+1$. Then the numerical solution $p_M(v,t)$ can be expressed as
\begin{equation}
\label{eq:41.1}
		p_M(v,t)=\sum_{k=1}^{2M+1}\hat{u}_k(t)\psi_k(v),
\end{equation}
where 
 \begin{equation}
 \label{eq:dis_basis}
    \left\{\psi_k\right\}^{2M+1}_{k=1}=\left\{g_1,h_0^L,h_1^L,\cdots,h_{M-1}^L,h_0^R,h_1^R,\cdots,h_{M-1}^R\right\}.
\end{equation}

According to the spectral-Galerkin method in \cite{shen1994efficient}, the variational formulation obtained from \eqref{eq:variational formula} is as bellow 
\begin{equation}
\label{eq:galerkin approximation}
    \begin{cases}
        \text{Find $p_M(v,t)\in W_M$ such that}\\
	\begin{aligned}\
        \mathcal{B}(p_M,\phi)+a(N(t))\partial_vp_M(V_F,t)\left[\phi(V_R,t)-\phi(V_F,t)\right]=0,\quad \forall \phi \in W_M,
	\end{aligned}
    \end{cases}
\end{equation}
where $W_M$ is the subspace of $W$. Substituting the numerical solution \eqref{eq:41.1} to the variational formulation \eqref{eq:galerkin approximation}, we obtain the semi-discrete form
\begin{equation}
\label{eq:42.1}
		H\frac{\mathrm{d}\hat{\boldsymbol{u}}}{\mathrm{d}t}+A\hat{\boldsymbol{u}}-bN(t)B\hat{\boldsymbol{u}}+a(N(t))C\hat{\boldsymbol{u}}+a(N(t))D\hat{\boldsymbol{u}}=0,
\end{equation}
where
\begin{equation}
\label{eq:43.1}
\begin{split}
&\hat{\boldsymbol{u}}=\big( \hat{u}_1(t),\hat{u}_2(t),\hat{u}_3(t),...,\hat{u}_{2M+1}(t) \big)^T,\\
&H_{jk}=\int_{-\infty}^{V_F} \psi_k\psi_j\mathrm{d}v,\quad A_{jk}=\int_{-\infty}^{V_F} v\psi_k\partial_v\psi_j\mathrm{d}v,\\
&B_{jk}=\int_{-\infty}^{V_F} \psi_k\partial_v\psi_j\mathrm{d}v,\quad C_{jk}=\int_{-\infty}^{V_F} \partial_v\psi_k\partial_v\psi_j\mathrm{d}v,\\
&D_{jk}=\partial_v\psi_k(V_F)\left[\psi_j(V_R)-\psi_j(V_F)\right],\quad j,k=1,2,\cdots,2M+1.
\end{split}
\end{equation}
The mean firing rate $N(t)$ is obtained by 
\begin{equation}
    \label{eq:N_one_pop}
    N(t)=-a(N(t))\sum_{k=1}^{2M+1} \hat{u}_k(t)\partial_v\psi_k(V_F).
\end{equation}
The initial coefficients can be determined by the projection method,
\begin{equation}
\label{eq:numerical initial}
    \int_{-\infty}^{V_F}\sum_{k=1}^{2M+1}\hat{u}_k(0)\psi_k(v)\psi_j(v)\mathrm{d}v=\int_{-\infty}^{V_F}p^0(v)\psi_j(v)\mathrm{d}v,\quad j=1,2,\cdots,2M+1.
\end{equation}

To achieve a balance between stability and efficiency in temporal discretization, we employ a semi-implicit scheme, a technique well-suited for spectral methods. The linear part of the equation is treated implicitly to enhance stability, while the nonlinear part is treated explicitly to avoid the complexity of solving nonlinear equations. Then the fully discretized form of the variational formulation \eqref{eq:galerkin approximation} is as follows
\begin{equation}
\label{eq:44}
		\left(\frac{1}{\Delta t}H+A-bN^nB+a^nC+a^nD\right)\hat{\boldsymbol{u}}^{n+1}=\frac{1}{\Delta t}H\hat{\boldsymbol{u}}^n,
\end{equation}
where 
\begin{equation}
    N^n=-a(N^n)\sum_{k=1}^{2M+1} \hat{u}^n_k\partial_v\psi_k(V_F).
\end{equation}

\subsection{Consistency of the numerical scheme}

In the above section, we have derived the variational formulation and developed a numerical scheme for the Fokker-Planck equation \eqref{eq:1}. Note that the exact solution does not satisfy the numerical scheme \eqref{eq:galerkin approximation} when it is truncated to the same order. However, we shall show that it satisfies the numerical scheme \eqref{eq:galerkin approximation} when the dimension of the approximation
space goes to infinity.

It is clear that 
    \begin{equation}
    \label{c1}
        p(v,t)=\sum_{k=1}^{\infty}u_k(t)\psi_k(v)
    \end{equation}
    satisfies the variational formulation \eqref{eq:variational formula}. Define
            \begin{align}
                \mathcal{L}(p,\phi) =
                \mathcal{B}(p,\phi)+a(N(t))\partial _v p(V_F,t)\left[ \phi(V_R)-\phi(V_F) \right],
        \end{align}
    we have 
    \begin{equation}
        \mathcal{L}(p,\phi)=0,\quad \forall \phi \in W.
    \end{equation}
    From the variational formulation \eqref{eq:galerkin approximation}, we have
    \begin{equation}
        \mathcal{L}(p_{M},\phi)=0,\quad \forall \phi \in W_{M},
    \end{equation}
    where
    \begin{equation}
        p_M=\sum_{k=1}^{2M+1}\hat{u}_k(t)\psi_k(v).
    \end{equation}
    Obviously, provided that we truncate the exact solution $p(v,t)$ in \eqref{c1} to 
    \begin{equation}
    \label{c2}
        \Tilde{p}_M=\sum_{k=1}^{2M+1}u_k(t)\psi_k(v),
    \end{equation}
    then
    \begin{equation}
        p(v,t)=\sum_{k=1}^{\infty}u_k(t)\psi_k(v)=\Tilde{p}_M+\sum_{k=2M+2}^{\infty}u_k(t)\psi_k(v).
    \end{equation}
    It is straightforward that $\Tilde{p}_M$ does not satisfy the variational formulation \eqref{eq:galerkin approximation}. Specifically, the coefficients $u_k$ from  $\Tilde{p}_M$ are the same as those in \eqref{c1} for the first $2M+1$ terms, but differ from the coefficients $\hat{u}_k$ calculated from the variational formulation \eqref{eq:galerkin approximation}, as shown in \eqref{eq:41.1}. Therefore, 
    \begin{equation}
        \mathcal{L}(\Tilde{p}_{M},\phi)\ne 0,\quad \forall \phi \in W_{M}.
    \end{equation}
    However, we claim that $\Tilde{p}_M$ satisfies the variational formulation \eqref{eq:galerkin approximation} in the limit as $M\to+\infty$. Namely, for any fixed $M_1>0$ and any test function $\phi_i \in W_{M_1}$, $M>M_1$
    \begin{equation}
    \label{eq:consistency}
           \begin{aligned}
                \mathcal{L}(\Tilde{p}_{M},\phi_i)=
                \mathcal{B}(\Tilde{p}_M,\phi_i)+\Tilde{a}(\Tilde{N}(t))\partial_v\Tilde{p}_{M}(V_F,t)\left[\phi_i(V_R,t)-\phi_i(V_F,t)\right]\to0,
           \end{aligned}
    \end{equation}
    where
    \begin{equation}
    \label{eq:consisency coeff}
		\begin{cases}
			\Tilde{N}(t)=-\Tilde{a}(\Tilde{N}(t))\partial_v \Tilde{p}_M(V_F,t), \\
			\Tilde{h}(v,\Tilde{N}(t))=-v+b\Tilde{N}(t),\\
			\Tilde{a}(\Tilde{N}(t))=a_0+a_1\Tilde{N}(t).
		\end{cases}
    \end{equation}
    Let us denote the above assertion as the consistency of the numerical scheme \eqref{eq:galerkin approximation}. 
    Consequently, we can summarize the above discussion as the following theorem.
    
\begin{theorem}[Consistency]
    Assume that
    \begin{equation}
    \label{eq:c1}
    \begin{aligned}
        \lim_{M\to+\infty}\left\| p-\Tilde{p}_M \right\|_{L^2(-\infty,V_F]}=0,\\
        \lim_{M\to+\infty}\left\| \partial_t(p-\Tilde{p}_M) \right\|_{L^2(-\infty,V_F]}=0,\\
        \lim_{M\to+\infty}\left\| \partial_v(p-\Tilde{p}_M) \right\|_{L^{\infty}(V_R,V_F]}=0,
    \end{aligned}
    \end{equation}
    then the numerical scheme \eqref{eq:galerkin approximation} is consistent with the differential equation \eqref{eq:1}, which means that \eqref{eq:consistency}-\eqref{eq:consisency coeff} hold as $M \to\infty$.
\end{theorem}

\begin{proof}
    Denoting $\sum\limits_{k=2M+2}^{\infty}u_k(t)\psi_k(v)$ as $R_M(v,t)$, then we rewrite $p(v,t)$ as
    \begin{equation}
    \label{eq:c2}
        p(v,t)=\Tilde{p}_M(v,t)+R_M(v,t).
    \end{equation}
  Here, our goal is to show 
    \begin{equation}
    \label{eq:51}
    \lim_{M\to+\infty}\mathcal{L}(\Tilde{p}_{M},\phi_i)=0,\quad \phi_i\in W_{M_1}.
    \end{equation}
    First, we observe that 
    \begin{equation}
        \mathcal{L}(p,\phi_i)=0, \quad \phi_i\in W_{M_1},
    \end{equation}
    because of $W_{M_1}\subseteq W$ from the variational formulation \eqref{eq:variational formula}. By subtracting this identity, we get
    \begin{equation}
        \label{eq:c4}
        \begin{aligned}
            \mathcal{L}(\Tilde{p}_{M},\phi_i)=&\mathcal{L}(\Tilde{p}_{M},\phi_i) - \mathcal{L}(p,\phi_i)\\
        =&\int_{-\infty}^{V_F} \partial_t(\Tilde{p}_M-p)\phi_i+(hp-\Tilde{h}\Tilde{p}_M)\partial_v\phi_i+(\Tilde{a}\partial_v\Tilde{p}_M-a\partial_vp)\partial_v\phi_i \mathrm{d}v \\ 
            &+\big[\Tilde{a}\partial_v\Tilde{p}_M(V_F,t)-a\partial_vp(V_F,t)\big]\big[\phi_i(V_R)-\phi_i(V_F)\big].
        \end{aligned}
    \end{equation}
    To facilitate the comparison of the terms in \eqref{eq:c4}, we introduce intermediate terms as follows
    \begin{equation}
    \label{eq:c6}
    \begin{aligned}
        \mathcal{L}(\Tilde{p}_{M},\phi_i) =& \int_{-\infty}^{V_F} \partial_t(\Tilde{p}_M-p)\phi_i\,\mathrm{d}v 
        + \int_{-\infty}^{V_F} \big(hp - \Tilde{h} p + \Tilde{h} p - \Tilde{h} \Tilde{p}_M\big) \partial_v \phi_i \,\mathrm{d}v \\
        &+ \int_{-\infty}^{V_F} \big(\Tilde{a} \partial_v \Tilde{p}_M - \Tilde{a} \partial_v p + \Tilde{a} \partial_v p - a \partial_v p\big) \partial_v \phi_i \,\mathrm{d}v \\
        &+ \big[\Tilde{a} \partial_v \Tilde{p}_M(V_F,t) - \Tilde{a} \partial_v p(V_F,t) + \Tilde{a} \partial_v p(V_F,t) - a \partial_v p(V_F,t)\big] \left[\phi_i(V_R) - \phi_i(V_F)\right].
    \end{aligned}
\end{equation}
    After simplification, we group the consistency error terms and obtain
    \begin{equation}
    \label{eq:c7.1}
    \begin{aligned}
        \mathcal{L}(\Tilde{p}_{M}, \phi_i) = \sum_{k=1}^{7}I_k,
    \end{aligned}
\end{equation}
where   
    \begin{equation}
    \label{eq:c7}
    \begin{aligned}
        I_1&=\int_{-\infty}^{V_F} \partial_t(\Tilde{p}_M - p) \phi_i \mathrm{d}v ,\quad
        I_2= \int_{-\infty}^{V_F} (h - \Tilde{h}) p \partial_v \phi_i \mathrm{d}v ,\quad
        I_3= \int_{-\infty}^{V_F} \Tilde{h} (p - \Tilde{p}_M) \partial_v \phi_i \mathrm{d}v ,\\
        I_4&= \int_{-\infty}^{V_F} \Tilde{a} (\partial_v \Tilde{p}_M - \partial_v p) \partial_v \phi_i  \mathrm{d}v ,\quad
        I_5= \int_{-\infty}^{V_F} (\Tilde{a} - a) \partial_v p \partial_v \phi_i \mathrm{d}v ,\\
        I_6&= \Tilde{a} \big[\partial_v \Tilde{p}_M(V_F, t) - \partial_v p(V_F, t)\big] \left[\phi_i(V_R) - \phi_i(V_F)\right] ,\\
        I_7&= (\Tilde{a} - a) \partial_v p(V_F, t) \left[\phi_i(V_R) - \phi_i(V_F)\right].
    \end{aligned}
\end{equation}
    Since
    \begin{equation}
        \begin{aligned}
            N(t)=-a\partial_vp(V_F,t)=-\big(a_0+a_1N(t)\big)\partial_vp(V_F,t),
        \end{aligned}
    \end{equation}
    we have
    \begin{equation}
    \label{eq:c8}
        N(t)=\frac{-a_0\partial_vp(V_F,t)}{1+a_1\partial_vp(V_F,t)}.
    \end{equation}
    Therefore, defining 
    \begin{equation}
    \label{eq:c9}
        \Tilde{N}(t)=\frac{-a_0\partial_v\Tilde{p}_M(V_F,t)}{1+a_1\partial_v\Tilde{p}_M(V_F,t)}.
    \end{equation}
it holds that 
    \begin{equation}
    \label{eq:c10}
        \begin{aligned}
            h-\Tilde{h}&=\frac{ba_0\partial_v\Tilde{p}_M(V_F,t)-\partial_vp(V_F,t)}{\big(1+a_1\partial_vp(V_F,t)\big)\big(1+a_1\partial_v\Tilde{p}_M(V_F,t)\big)},\\
            \Tilde{a}-a&=\frac{a_1a_0\partial_vp(V_F,t)-\partial_v\Tilde{p}_M(V_F,t)}{\big(1+a_1\partial_vp(V_F,t)\big)\big(1+a_1\partial_v\Tilde{p}_M(V_F,t)\big)}.
        \end{aligned}
    \end{equation}
    Hence, from the error estimation \eqref{eq:c1}, there exists $M_0>0$ and $M_0^{'}>0$ such that
    \begin{equation}
        \begin{aligned}
        \left\| h-\Tilde{h} \right\|_{L^{\infty}(V_R,V_F]}&\le M_0 \left\| (\partial_v(\Tilde{p}_M-p)(V_F,t) \right\|_{L^{\infty}(V_R,V_F]},\\
        \left\| \Tilde{a}-a \right\|_{L^{\infty}(V_R,V_F]}&\le M^{'}_0 \left\| (\partial_v(p-\Tilde{p}_M)(V_F,t) \right\|_{L^{\infty}(V_R,V_F]}.
        \end{aligned}
    \end{equation}
    Namely,
    \begin{equation}
    \label{eq:c13}
        \begin{aligned}
            \lim_{M\to+\infty}\left\| h-\Tilde{h} \right\|_{L^{\infty}(V_R,V_F]}&=0,\\
            \lim_{M\to+\infty}\left\| \Tilde{a}-a \right\|_{L^{\infty}(V_R,V_F]}&=0.
        \end{aligned}
    \end{equation}
    Hence, it is easy to prove that
    \begin{equation}
    \begin{aligned}
        \label{eq:c15}
        &|I_6|=\big|\Tilde{a}\left[\partial_v\Tilde{p}_M(V_F,t)-\partial_vp(V_F,t)\right]\big|\le  |\Tilde{a}|\left\| (\partial_v(\Tilde{p}_M-p)(V_F,t) \right\|_{L^{\infty}(V_R,V_F]},\\
        &|I_7|=\big|(\Tilde{a}-a)\partial_vp(V_F,t)\big|\le \left\| \Tilde{a}-a \right\|_{L^{\infty}(V_R,V_F]}\big|\partial_vp(V_F,t)\big|,
    \end{aligned}
    \end{equation}
    which implies that
    \begin{equation}
        \label{eq:I67_to_0}
        \lim_{M\to+\infty}|I_6|=0,\quad \lim_{M\to+\infty}|I_7|=0.
    \end{equation}
    From the definition of weak solution, we can know that $\phi_i(v)\in H^1((-\infty, V_F])$. From the Holder's inequality, we get
    \begin{equation}
    \label{c17}
        \begin{aligned}
            |I_1|&=\left|\int_{-\infty}^{V_F} \partial_t(\Tilde{p}_M-p)\phi_i\mathrm{d}v\right|
            \le \left(\int_{-\infty}^{V_F} \big(\partial_t(\Tilde{p}_M-p)\big)^2\mathrm{d}v\right)^{\frac{1}{2}}\left(\int_{-\infty}^{V_F}\phi_i^2\mathrm{d}v\right)^{\frac{1}{2}}\\
            &= \left\| \partial_t(\Tilde{p}_M-p) \right\|_{L^2(-\infty,V_F]}\left(\int_{-\infty}^{V_F}\phi_i^2\mathrm{d}v\right)^{\frac{1}{2}}.
        \end{aligned}
    \end{equation}
    Therefore, we obtain
    \begin{equation}
        \label{eq:I1_to_0}
        \lim_{M\to+\infty}|I_1|=0.
    \end{equation}
    Because of $\phi_i(v)\in H^1((-\infty,V_F])$, $\partial_v\phi_i\in L^2((-\infty,V_F])$. Then from \eqref{eq:c13} and Holder's inequality, we obtain
    \begin{equation}
        \label{eq:c18}
        \begin{aligned}
            |I_2|&=\left|\int_{-\infty}^{V_F}[(h-\Tilde{h})p]\partial_v\phi_i\mathrm{d}v\right|
            \le \left\| h-\Tilde{h} \right\|_{L^{\infty}(V_R,V_F]}\int_{-\infty}^{V_F}p\partial_v\phi_i\mathrm{d}v
        \end{aligned}
    \end{equation}
    and 
    \begin{equation}
        \label{eq:I2_to_0}
        \lim_{M\to+\infty}|I_2|=0.
    \end{equation}
    From the definition of weak solution, we know that $v\partial_v\phi_i\in H^{1}((-\infty, V_F])$. Therefore, applying the Holder's inequality, we get that
    \begin{equation}
    \label{eq:c19}
        \begin{aligned}
           |I_3|&= \left|\int_{-\infty}^{V_F}\Tilde{h}(p-\Tilde{p}_M)\partial_v\phi_i\mathrm{d}v\right|
           =\left|\int_{-\infty}^{V_F}\left(-v\partial_v\phi_i+b\Tilde{N}(t)\partial_v\phi_i\right)(p-\Tilde{p}_M)\mathrm{d}v\right|\\
            &\le \left(\int_{-\infty}^{V_F}\left(-v\partial_v\phi_i+b\Tilde{N}(t)\partial_v\phi_i\right)^2\mathrm{d}v\right)^{\frac{1}{2}} \left(\int_{-\infty}^{V_F}(p-\Tilde{p}_M)^2\mathrm{d}v\right)^{\frac{1}{2}},
        \end{aligned}
    \end{equation}
    which can be written as
    \begin{equation}
    \label{eq:c20}
        \begin{aligned}
        |I_3|
        &\le \left(\int_{-\infty}^{V_F}\left(-v\partial_v\phi_i+b\Tilde{N}(t)\partial_v\phi_i\right)^2\mathrm{d}v\right)^{\frac{1}{2}} \left\|p-\Tilde{p}_M \right\|_{L^2(-\infty,V_F]}.
        \end{aligned}
    \end{equation}
    Thus, we have
    \begin{equation}
        \label{eq:I3_to_0}
        \lim_{M\to+\infty}|I_3|=0.
    \end{equation}
    Similarly, we can obtain the following conclusion
    \begin{equation}
        \label{eq:c21}
        \begin{aligned}
          |I_4|&=  \left|\int_{-\infty}^{V_F}\Tilde{a}(\partial_v\Tilde{p}_M-\partial_vp)\partial_v\phi_i\mathrm{d}v\right|
            = \left|\int_{-\infty}^{V_F}(a_0+a_1\Tilde{N}(t))\partial_v(\Tilde{p}_M-p)\partial_v\phi_i\mathrm{d}v\right|\\
            & \le \left(\int_{-\infty}^{V_F}\left((a_0+a_1\Tilde{N}(t))\partial_v\phi_i\right)^2\mathrm{d}v\right)^{\frac{1}{2}} \left\|\partial_v(\Tilde{p}_M-p) \right\|_{L^2(-\infty,V_F]},
             \\
           |I_5|&= \left|\int_{-\infty}^{V_F}(\Tilde{a}-a)\partial_vp\partial_v\phi_i\mathrm{d}v\right| 
           =(\Tilde{a}-a)\int_{-\infty}^{V_F}\partial_vp\partial_v\phi_i\mathrm{d}v
            \le \left\| \Tilde{a}-a \right\|_{L^{\infty}(V_R,V_F]}\int_{-\infty}^{V_F}\partial_vp\partial_v\phi_i\mathrm{d}v.
        \end{aligned}
    \end{equation}
    Therefore we get
    \begin{equation}
        \label{eq:I45_to_0}
        \lim_{M\to+\infty}|I_4|=0, \quad \lim_{M\to+\infty}|I_5|=0.
    \end{equation}
    In conclusion, we have
    \begin{equation}
        \label{eq:c23}
        \begin{aligned}
            \lim_{M\to+\infty}I_k=0,\quad k=1,2,\cdots,7. 
        \end{aligned}
    \end{equation}
    Accordingly, substituting \eqref{eq:c23} to \eqref{eq:c7}, we obtain
    \begin{equation}
    \label{eq:c24}
        \begin{split}
            \lim_{M\to+\infty}\mathcal{L}(\Tilde{p}_{M},\phi_i)=\lim_{M\to+\infty}\sum_{k=1}^{7}I_k=0. 
        \end{split}
    \end{equation}
    Namely, the numerical scheme \eqref{eq:galerkin approximation} is consistent with the Fokker-Planck equation \eqref{eq:1}. The proof is completed.

\end{proof}

\section{Model extensions}\label{sec:two populations}
In the previous section, we introduced a fully discrete numerical scheme for the Fokker-Planck equation \eqref{eq:1}. Thanks to the flexibility of the LLSGM, this method is readily applicable to more complex models. In this section, we consider a more realistic NNLIF model (see e.g. \cite{caceres2018Towards}), featuring two populations, synaptic delays, and refractory periods, and its numerical simulation is more challenging. To verify the flexibility of the proposed numerical method, we explore the extended models from a numerical perspective.

\subsection{Model introduction}
\label{sec:two-pop-model}
In the one-population model, we examine the simplest scenario in which all neurons are considered either excitatory or inhibitory. However, in large-scale neuron networks, the dynamical behavior becomes more intricate due to the interplay between two distinct populations of neurons, excitatory and inhibitory. To accurately capture the dynamics of a neuron network, the model \eqref{eq:1} has been extended to incorporate both excitatory populations (E-P) and inhibitory populations (I-P). The extended system now comprises two probability density functions, $p_E(v,t)$ for the excitatory population and $p_I(v,t)$ for the inhibitory population, each governed by an analogous Fokker-Planck equation \cite{caceres2016blow} with \eqref{eq:1}, as follows
\begin{equation}
    \label{eq:two_pop_with_N}
    \partial _t p_{\alpha} +\partial _v \big[h^{\alpha}\big(v, N_E\left(t\right), N_I\left(t\right)\big) p_{\alpha}(v,t)\big]-a_{\alpha}\big(N_E\left(t\right), N_I\left(t\right)\big)\partial _{vv} p_{\alpha}(v,t) = N_{\alpha}(t)\delta(v-V_R),
\end{equation}
where 
\begin{equation}
    \label{eq:two_pop_N}
    N_\alpha(t)=-a_\alpha\left(N_E(t), N_I(t)\right) \partial_v p_\alpha\left(V_F, t\right), \quad \alpha, \beta =E, I
\end{equation}
is the mean firing rate of the population $\alpha$. The drift and diffusion coefficients are
\begin{equation}
\label{eq:two_pop_para}
\begin{aligned}
h^\alpha\left( v, N_E, N_I \right) & =-v+b_E^\alpha N_E-b_I^\alpha N_I+\left(b_E^\alpha-b_E^E\right) \nu_{E, \text {ext}}, \\
a_\alpha\left(N_E, N_I\right) & =\mathrm{d}_E^\alpha \nu_{E, \text {ext}}+\mathrm{d}_E^\alpha N_E+\mathrm{d}_I^\alpha N_I,
\end{aligned}
\end{equation}
where $\nu_{E, \text {ext}} \ge 0$ describes the external inputs originating from excitatory neurons.
The synaptic strength between the two populations is governed by the connectivity parameters $b^{\alpha}_{\beta}\ge0$ and $d^{\alpha}_{\beta}\ge0$, which correspond to the parameters $b$ and $a_1$ in \eqref{eq:1}, respectively. Specifically, these parameters describe the strength of synaptic connections from neurons in population $\beta$ to neurons in population $\alpha$. In contrast, the connectivity parameter $b$ in the one-population model primarily represents the interactions between neurons within the same population, without the distinction between excitatory and inhibitory neurons.

In addition to the interactions between the two populations, synaptic delays, which describe the transmission delay of a spike, are introduced to enhance the realism of the model. When a neuron in population $\beta$ fires, its effect on neurons in population $\alpha$ is delayed by a constant time interval $D^{\alpha}_{\beta}$. The Fokker-Planck equations with the synaptic delays are expressed as
\begin{equation}
    \label{eq:two_pop_with_DN}
    \partial _t p_{\alpha} +\partial _v \big[h^{\alpha}\big(v, N_E\left(t-D_E^{\alpha}\right), N_I\left(t-D_I^{\alpha}\right)\big) p_{\alpha}\big]-a_{\alpha}\big(N_E\left(t-D_E^{\alpha}\right), N_I\left(t-D_I^{\alpha}\right)\big)\partial _{vv} p_{\alpha} = N_{\alpha}(t)\delta(v-V_R).
\end{equation}

After firing, neurons enter a refractory period, during which they are temporarily unresponsive to inputs. Over time, neurons gradually recover from their refractory state, becoming responsive again to inputs. The mean recovery rate from the refractory period is denoted by $M_{\alpha}(t)$. This mechanism is modeled by the following differential equations that describe the probability of neurons being in the refractory state $R_{\alpha}$
\begin{equation}
    \label{eq:refractory_R}
    \begin{aligned}
        \displaystyle &\frac{\mathrm{d}R_{\alpha}(t)}{\mathrm{d}t}=N_{\alpha}(t)-M_{\alpha}(t).\\
    \end{aligned}
\end{equation} 
There are multiple choices for $M_{\alpha}(t)$ \cite{brunel2000dynamics,caceres2014beyond}. In this section, we consider
\begin{equation}
    \label{eq:M_equation}
    M_{\alpha}(t)=\frac{R_{\alpha}(t)}{\tau_{\alpha}},
\end{equation}
where $\tau_{\alpha}$ measures the mean duration of the refractory period. Therefore, the corresponding Fokker-Planck equation, which includes the recovery mechanism, is given by
\begin{equation}
    \label{eq:two_pop_with_M}
    \partial _t p_{\alpha} +\partial _v \big[h^{\alpha}\big(v, N_E\left(t-D_E^{\alpha}\right), N_I\left(t-D_I^{\alpha}\right)\big) p_{\alpha}\big]-a_{\alpha}\big(N_E\left(t-D_E^{\alpha}\right), N_I\left(t-D_I^{\alpha}\right)\big)\partial _{vv} p_{\alpha} = M_{\alpha}(t)\delta(v-V_R).
\end{equation}
The source term on the right-hand side represents the fact that the membrane potential is immediately reset to $V_R$ after the neuron fires, and it stays in the refractory period for a duration.

Compared to the previous NNLIF model \eqref{eq:1}, this mesoscopic model includes two Fokker-Planck equations that govern the evolution of the probability density functions \(p_\alpha(v,t)\), along with two ordinary differential equations (ODEs) that describe the evolution of the refractory states \(R_\alpha\) of the neurons. By setting the same equation on $(-\infty, V_R)\cup(V_R, V_F]$ and incorporating the jump conditions, Dirichlet boundary conditions, and initial conditions, we have this nonlinear system \cite{Cáceres2018Analysis}
\begin{equation}
\label{eq:two population}
\begin{cases}
    \displaystyle \partial _t p_{\alpha} +\partial _v \big[h^{\alpha}\big(v, N_E\left(t-D_E^{\alpha}\right), N_I\left(t-D_I^{\alpha}\right)\big) p_{\alpha}\big]-a_{\alpha}\big(N_E\left(t-D_E^{\alpha}\right), N_I\left(t-D_I^{\alpha}\right)\big)\partial _{vv} p_{\alpha} =0,\\
\begin{aligned}
    \displaystyle &\frac{\mathrm{d}R_{\alpha}(t)}{\mathrm{d}t}=N_{\alpha}(t)-M_{\alpha}(t),\\ &N_\alpha(t)=-a_\alpha\big(N_E(t-D_E^{\alpha}), N_I(t-D_I^{\alpha})\big) \partial_v p_\alpha\left(V_F, t\right),
\end{aligned}\\
    \begin{aligned}
	\displaystyle p_\alpha(v,0)=p^0_\alpha(v)\ge0,\, R_\alpha(0)=R^0_\alpha\ge0, 
    \end{aligned}\\
    \displaystyle p_\alpha(-\infty,t)=p_\alpha(V_F,t)=0,\,p_{\alpha}(V_R^-,t)=p_{\alpha}(V_R^+,t),\, \partial_vp_{\alpha}(V_R^-,t)-\partial_vp_{\alpha}(V_R^+,t)=\frac{M_{\alpha}(t)}{a_{\alpha}(N_E(t),N_I(t))}.
\end{cases}
\end{equation}

Since the number of neurons remains unchanged, the total mass should be conserved, as shown below
\begin{equation}
\label{eq:two_mass_conservation}
    \int_{-\infty}^{V_F} p_\alpha(v, t) \mathrm{d} v+R_\alpha(t)=\int_{-\infty}^{V_F} p_\alpha^0(v) \mathrm{d} v+R_\alpha^0=1, \quad \forall t \geq 0, \quad \alpha=E, I .
\end{equation}
Despite previous theoretical and numerical studies on the model by \cite{caceres2018Towards, Cáceres2018Analysis, sharma2020discontinuous}, several properties remain unexplored. This study further investigates the periodic solution phenomena in this model.

\subsection{Numerical scheme}
\label{sec:two-pop-scheme}
We now elaborate the numerical scheme for $\eqref{eq:two population}$, employing the spectral method to discretize space and a semi-implicit format for time discretization of the Fokker-Planck equations. Concurrently, the ODEs are solved using the forward Euler method.

Thanks to the derivation in Sec. \ref{sec:weak_form}, the variational formulation of equation \eqref{eq:two population} can be naturally obtained as follows
\begin{equation}
\label{eq:two_weakform}
\begin{cases}
    \text{Given $p^0_\alpha(v) \in W$, find $p_{\alpha}(v,t)\in W, \alpha= E, I$, such that}\\
    \begin{aligned}
        \int_{-\infty}^{V_F} \Big(\partial _t p_{\alpha}(v,t)\phi(v) &-h^{\alpha}p_{\alpha}(v,t)\partial _v \phi(v) +a_{\alpha}\partial _v p_{\alpha}(v,t)\partial _v\phi(v) \Big)\mathrm{d}v \\
        &-a_{\alpha}\partial _v p_{\alpha}(V_F,t)\phi(V_F)-M_{\alpha}(t)\phi(V_R)=0,\quad \forall \phi \in W,
    \end{aligned}\\
    \displaystyle \frac{\mathrm{d}R_{\alpha}(t)}{\mathrm{d}t}=N_{\alpha}(t)-M_{\alpha}(t),\\
    p_\alpha(v,0)=p^0_\alpha(v),\, R_\alpha(0)=R^0_\alpha.
\end{cases}
\end{equation}
Here $W$ is the trial function space. Similar to the fully discrete scheme in Sec. \ref{sec:Fully_dis}, we first truncate the trial function space $W$ to obtain $W_M$, then the numerical solution $p_{\alpha,M}$ is expanded as
\begin{equation}
    p_{\alpha,M}(v,t)=\sum_{k=1}^{2M+1}u_{\alpha,k}(t)\psi_k(v), \quad \forall \psi_k \in W_M.
\end{equation}
The initial condition for the expansion coefficients $\{u_{\alpha,k}(0)\}_{k=1}^{2M+1}$ can be obtained by the least square approximation
\begin{equation}
    \int_{-\infty}^{V_F} \sum_{k=1}^{2M+1} u_{\alpha,k}(0)\psi_k(v)\psi_j(v) \mathrm{d}v = \int_{-\infty}^{V_F} p^0_{\alpha}(v) \psi_j(v)  \mathrm{d}v ,\quad j=1,2,\cdots,2M+1.
\end{equation}
Similarly, we choose the semi-implicit scheme in the time direction. Namely, the nonlinear part is treated explicitly and other parts are treated implicitly. 
Notice that, the drift coefficients $h_{\alpha}$ and diffusion coefficients $a_{\alpha}$ in the model involve synaptic delay terms related to historical information. So we define the delayed time index
\begin{equation}
    n_{\beta}^{\alpha} = \max\left(0,n - \frac{D_{\beta}^{\alpha}}{\Delta t}\right).
\end{equation}
Consequently, one should adjust the parameters to ensure that $n_{\beta}^{\alpha}$ is an integer in numerical experiments. After time discretization, we obtain the fully discrete scheme as follows, for $\alpha= E, I$,
 \begin{equation}
 \label{eq:two_fully_dis}
     \begin{cases}
        \vspace{5pt}
        \displaystyle \left(\frac{1}{\Delta t}H+A-V_{\alpha}(N_E^{n_E^\alpha},N_I^{n_I^\alpha})B+a_{\alpha}(N_E^{n_E^\alpha},N_I^{n_I^\alpha})C+a_{\alpha}(N_E^{n_E^\alpha},N_I^{n_I^\alpha})G\right)\boldsymbol{u}^{n+1}_{\alpha}=\frac{1}{\Delta t}H\boldsymbol{u}^{n}_{\alpha}+M^n_{\alpha}F,\\
        R^{n+1}_\alpha=R^{n}_\alpha+\Delta t (N^{n}_\alpha-M^{n}_\alpha),
     \end{cases}
 \end{equation}
 where 
 \begin{equation}
 \begin{aligned}
     &\boldsymbol{u}^{n}_{\alpha}=\big( u_{\alpha,1}(t^n),u_{\alpha,2}(t^n),\cdots,u_{\alpha,2M+1}(t^n) \big)^T,\\
     &R^n_{\alpha} = R_{\alpha}(t^n),\\
     &M_{\alpha}^n=\frac{R_{\alpha}^n}{\tau_{\alpha}},\\
     &G_{jk}=\partial_v\psi_k(V_F)\psi_j(V_F),\quad j,k=1,2,\cdots,2M+1,\\
     &F=(\psi_1(V_R),\psi_2(V_R),\cdots,\psi_{2M+1}(V_R))^T.
 \end{aligned}
 \end{equation}
 and the matrix $H,A,B,C$ are defined in \eqref{eq:43.1}. The mean firing rate can be obtained by solving the following equations
 \begin{equation}
 \begin{aligned}
     &N^{n}_{E}=-a_{E}(N_E^{n_E^{\alpha}},N_I^{n_I^{\alpha}})\sum_{k=1}^{2M+1}u^{n}_{E,k}\partial_v\psi_k(V_F),\\
     &N^{n}_{I}=-a_{I}(N_E^{n_E^{\alpha}},N_I^{n_I^{\alpha}})\sum_{k=1}^{2M+1}u^{n}_{I,k}\partial_v\psi_k(V_F).\\
 \end{aligned}
\end{equation}
The synaptic delay terms in \eqref{eq:two_fully_dis} is mainly reflected in $V_\alpha,a_\alpha$ as
\begin{equation}
\begin{aligned}
    &V_{\alpha}(N_E^{n_E^\alpha},N_I^{n_I^\alpha})=b^\alpha_EN_E^{n_E^\alpha}-b^\alpha_IN_I^{n_I^\alpha}+(b^\alpha_E-b^E_E)\nu_{E,\text{ext}},\\
    &a_\alpha(N_E^{n_E^\alpha},N_I^{n_I^\alpha}) = \mathrm{d}_E^\alpha \nu_{E, \text {ext}}+\mathrm{d}_E^\alpha N_E^{n_E^\alpha}+\mathrm{d}_I^\alpha N_I^{n_I^\alpha}.
\end{aligned}
\end{equation}

\section{Numerical test} 
\label{sec:numerical_test}
In this section, several numerical tests are studied to validate LLSGM. First, the convergence rates of the numerical schemes proposed in \eqref{eq:44} for the one-population model and \eqref{eq:two_fully_dis} for the two-population model are studied, followed by a stability study in Sec. \ref{sec:accuracy_stability}. Then, the numerical examples for the efficiency of LLSGM are tested, including the comparison with the finite difference method \cite{he2022structure} and the spectral method \cite{zhang2024spectral} in Sec. \ref{sec:efficiency}. Moreover, the blow-up phenomena for both the one-population and two-population models are also investigated in Sec. \ref{sec:blow_up}. Finally, the numerical experiments are extended to the two-population model with synaptic delays and refractory states in Sec. \ref{sec:dely_refactory}, where transitions between periodic oscillations, steady states, and blow-up events are explored. 
		
In all numerical tests, $V_F$ and $V_R$ are fixed as $V_F = 2$ and $V_R = 1$, and the initial distribution is chosen as the Gauss function
\begin{equation}
    \label{eq:initial}
    p_G(v)=\frac{1}{\sqrt{2\pi}\sigma_0 M_0}\mathrm{e}^{-\frac{(v-v_0)^2}{2\sigma_0^2}},
\end{equation} 
where $v_0$ and $\sigma_0^2$ are the mean and variance respectively. Here, $M_0$ is a normalization factor, which makes $p_G(v)$ satisfy 
    \begin{equation}
    \label{eq:normal}
    \int_{-\infty}^{V_F} p_G(v) \mathrm{d}v=1.
\end{equation}
In the numerical tests, $v_0, \sigma_0$, and $M_0$ are set problem-dependent.

\subsection{Study of the order of accuracy and the stability test}
\label{sec:accuracy_stability}
In this section, the order of accuracy for LLSGM including the temporal and spatial convergence of the one-population model \eqref{eq:1} and \eqref{eq:1.1} is first tested, followed by the convergence order of LLSGM in both temporal and spatial discretizations for the two-population model. Finally, the stability of LLSGM for the one-population model is examined by varying the expansion number $M$ and the time-step length $\Delta t$.

\subsubsection{Order of accuracy for one-population model}
\label{sec:accuracy_one_pop}
In this section, the convergence order for LLSGM in temporal and spatial discretization is studied. The initial distribution in \eqref{eq:initial} is set as $v_0=-1$ and $\sigma_0^2=0.5$ with the parameters in \eqref{eq:1} chosen as $a_0=1, a_1=0.1$ and $b=0$.

The convergence of the explicit-implicit scheme \eqref{eq:44} in temporal discretization is tested by calculating the $L^2$ and $L^{\infty}$ errors between the numerical solution obtained with different time-step lengths and the reference solution. Here, the reference solution is obtained by the finite-difference method  (FDM) proposed in \cite{hu2021structure} with the mesh fine enough. In the simulation, the expansion number $M$ is set as $M = 16$, and the final time is $t = 0.2$. Tab. \ref{tab:ex1_time_first} shows the error calculated by \eqref{eq:ex1_error_1}
\begin{equation}
    \label{eq:ex1_error_1}
    O_{\Delta t,L^l}=\log_2\frac{\left\|p_{2\Delta t}-p_r\right\|_{l}}{\left\|p_{\Delta t}-p_r\right\|_l},
\end{equation}
where $p_r$ denotes the reference solution and $p_{\Delta t}$ is the numerical solution with the time-step length $\Delta t$. For the numerical solution, the time-step lengths are set as $\Delta t = 0.04, 0.02, 0.01$ and $0.005$. Tab. \ref{tab:ex1_time_first} indicates the first-order convergence of \eqref{eq:44} in the temporal direction.

\begin{table}[!hptb]
		\centering 
		\def\arraystretch{1.5}
		\scalebox{0.8}{
        		 \small
			\begin{tabular}{c|c|c|c|c}
				$\Delta t$ & $\left \|p_{\Delta t}- p_r  \right \|_2$ & $O_{\Delta t,L^2}$ & $\left \|p_{\Delta t}-p_r  \right \|_{\infty}$ &$O_{\Delta t,L^{\infty}}$ \\
				\hline
				0.04		& 4.58e-03	& ——	& 3.89e-03	& ——	\\
				0.02		& 2.36e-03	& 0.95	& 2.02e-03	& 0.95	\\
				0.01		& 1.20e-03	& 0.97	& 1.04e-03	& 0.97	\\
				0.005		& 6.09e-04	& 0.98	& 5.31e-04	& 0.96	\\
			\end{tabular}
		}
		\caption{(Order of accuracy in Sec. \ref{sec:accuracy_one_pop}) The first order convergence of the explicit-implicit scheme \eqref{eq:44} in the temporal discretization for the one-population model.}
		\label{tab:ex1_time_first}
\end{table}
	
Then the convergence of LLSGM \eqref{eq:44} in the spatial space is studied, where the $L^2$ error between the numerical solution obtained with different expansion numbers $M$ and the reference solution is calculated. Here, the same reference solution obtained by FDM is utilized. For the numerical solution, the time-step length is fixed as $\Delta t = 0.001$, and the expansion number varies from $M = 3$ to $12$. Here, we want to emphasize that the number of basis functions is $2M + 1$. The corresponding error is calculated as 
	\begin{equation}
		\label{eq:ex1_error_spatial}
		O_{M, L^2} = \ln  \left\| p_M-p_r \right\|_2,
	\end{equation}
where $p_M$ is the numerical solution with the different expansion number $M$. Fig. \ref{fig:ex1_spatial_one} shows the error obtained by \eqref{eq:ex1_error_spatial} with different $M$, where the numerical results obtained by odd and even $M$ are plotted respectively. In Fig. \ref{fig:ex1_spatial_one}, we observe a linear relation between the error \eqref{eq:ex1_error_spatial} and $M$, which indicates the spectral convergence of LLSGM. 
	
	\begin{figure}[!hptb]
		\centering
		\subfigure[odd $M$]{
			\includegraphics[width=0.45\textwidth]{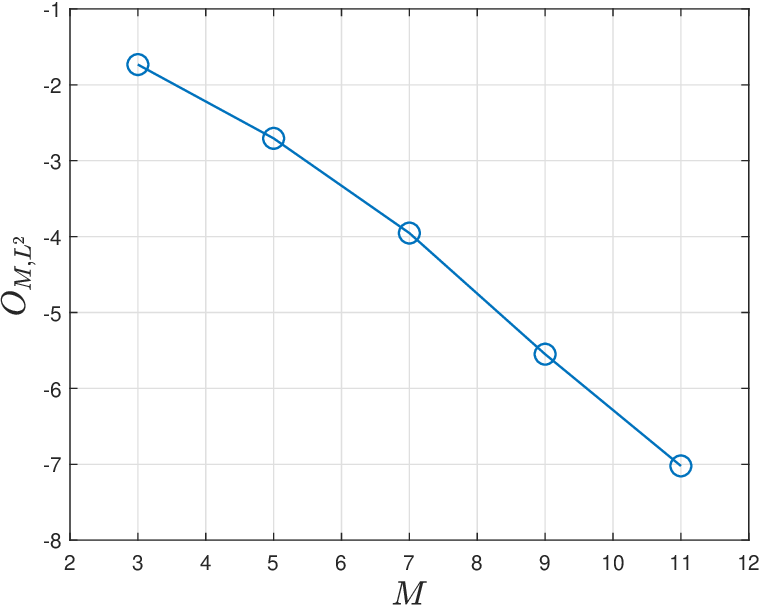}   
		}\hfill
		\subfigure[even $M$]{
			\includegraphics[width = 0.46\textwidth]{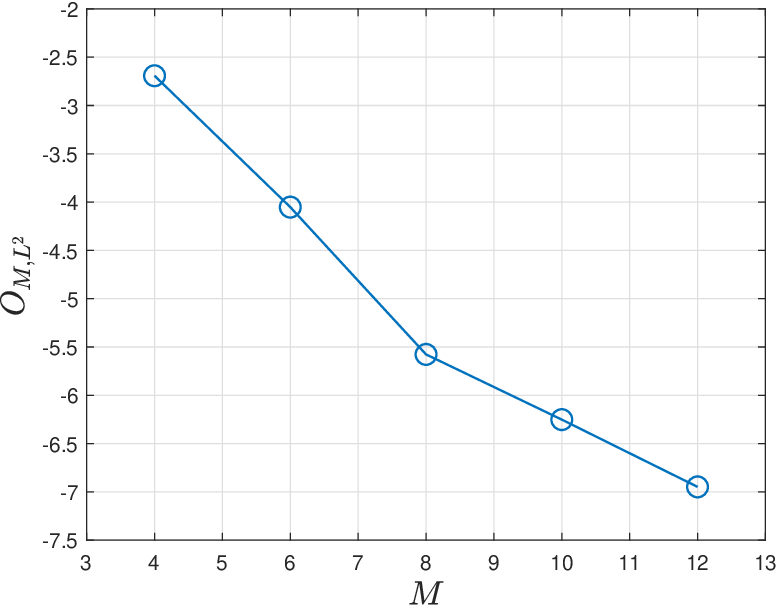}
		}
		\caption{(Order of accuracy in Sec. \ref{sec:accuracy_one_pop}) The spectral convergence of the explicit-implicit scheme \eqref{eq:44} in the spatial discretization for the one-population model. (a) $M$ is odd. (b) $M$ is even.}
		\label{fig:ex1_spatial_one}
	\end{figure}

\subsubsection{Order of accuracy for two-population model}
\label{sec:accuracy_two_pop}
   \begin{table}[!hptb]
		\centering
		\renewcommand{\arraystretch}{1.5}
		\scalebox{0.8}{
        		 \small
			\begin{tabular}{c|>{\centering\arraybackslash}p{3cm} >{\centering\arraybackslash}p{2cm}}
				\multirow{1}{*}{\text{diffusion coefficient}}          & $a$                  & $1$ \\ \hline
				\multirow{1}{*}{\text{external input}}                 & $\nu_{E,\text{ext}}$ & $0$ \\ \hline
				\multirow{1}{*}{\text{refractory periods}}              & $\tau_E,\tau_I$      & $0$ \\ \hline
				\multirow{2}{*}{\text{initial distribution}}   & $[v_0^{E}, (\sigma_0^{E})^2]$              &  $[-1, 0.5]$ \\
				& $[v_0^{I}, (\sigma_0^{I})^2]$              & $[0, 0.25]$   \\ \hline
				\multirow{1}{*}{\text{synaptic delays}}        & $D_E^E, D_E^I, D_I^E, D_I^I$               & $0$ \\ \hline
				\multirow{2}{*}{\text{connectivity parameters}}     & $[b_E^E,b_E^I]$                & $[0.5,0.5]$ \\  
				& $[b_I^E,b_I^I]$                & $[0.75,0.25]$ 
		\end{tabular}}
		\caption{(Order of accuracy in Sec. \ref{sec:accuracy_two_pop}) Computational parameters in \eqref{eq:two population} for the two-population model.}
		\label{tab:ex2_convergence_para}
	\end{table}

In this section, the convergence order of LLSGM applied to the two-population model is studied, focusing on both temporal and spatial discretization. For the two-population model, the same Gaussian function \eqref{eq:initial} is utilized for the initial distribution of both excitatory and inhibitory populations with the mean and variance denoted by $v_0^{\alpha}$ and $(\sigma_0^{\alpha})^2, \alpha=E, I$. Here, we set $M_{\alpha}=N_{\alpha}$ and the other parameters in \eqref{eq:two population} are shown in Tab. \ref{tab:ex2_convergence_para}.
	\begin{table}[!hptb]
	\centering 
	\def\arraystretch{1.5}
	\scalebox{0.8}{
    		 \small
		\begin{tabular}{c|c|c|c|c}
			Excitatory population  & $\left \|p^E_{\Delta t}- p^E_r  \right \|_2$ & $O^E_{\Delta t,L^2}$ & $\left \|p^E_{\Delta t}- p^E_r  \right \|_{\infty}$ &$O^E_{\Delta t,L^{\infty}}$ \\
			\hline
			$\Delta t=0.04$	    &4.16e-03 &  ——	&3.52e-03 & ——    \\
			$\Delta t=0.02$	    &2.15e-03 &0.95 &1.83e-03	&0.94 \\
			$\Delta t=0.01$	    &1.09e-03 &0.97 &9.46e-04	&0.95 \\
			$\Delta t=0.005$	&5.54e-04 &0.98 &4.81e-04	&0.98 \\
			\hline
			Inhibitory population  & $\left \|p^I_{\Delta t}- p^I_r  \right \|_2$ & $O^I_{\Delta t,L^2}$ & $\left \|p^I_{\Delta t}- p^I_r  \right \|_{\infty}$ &$O^I_{\Delta t,L^{\infty}}$ \\
			\hline
			$\Delta t=0.04$	    &1.02e-02 &  ——	&1.12e-02 & ——    \\
			$\Delta t=0.02$	    &5.28e-03 &0.96 &5.72e-03	&0.97 \\
			$\Delta t=0.01$	    &2.68e-03 &0.98 &2.87e-03	&0.99 \\
			$\Delta t=0.005$	&1.35e-03 &0.99 &1.42e-03	&1.01 \\
		\end{tabular}}
	\caption{(Order of accuracy in Sec. \ref{sec:accuracy_two_pop}) The first-order convergence of the scheme \eqref{eq:two_fully_dis} in the temporal discretization for the two-population model.}
	\label{tab:ex2_time}
\end{table}

Similar to the one-population model, to assess the temporal convergence of the numerical scheme \eqref{eq:two_fully_dis}, the $L^2$ and $L^{\infty}$ errors between the numerical solutions obtained with different time-step lengths and the reference solution are calculated. Here, the reference solution is obtained by the finite-difference method \cite{hu2021structure} with a sufficiently fine mesh. In the simulation, the expansion number is fixed as $M = 16$. The error calculated by \eqref{eq:ex1_error_1} for both the excitatory and inhibitory populations at $t = 0.2$ is shown in Tab. \ref{tab:ex2_time}, where the time-step length utilized to obtain the numerical solution is $\Delta t = 0.04, 0.02, 0.01$ and $0.005$. It clearly demonstrates the first-order convergence of the scheme \eqref{eq:two_fully_dis} in the temporal direction.

\begin{figure}[!hptb]
    \centering
    \subfigure[odd $M$]{
    \includegraphics[width=0.45\textwidth]{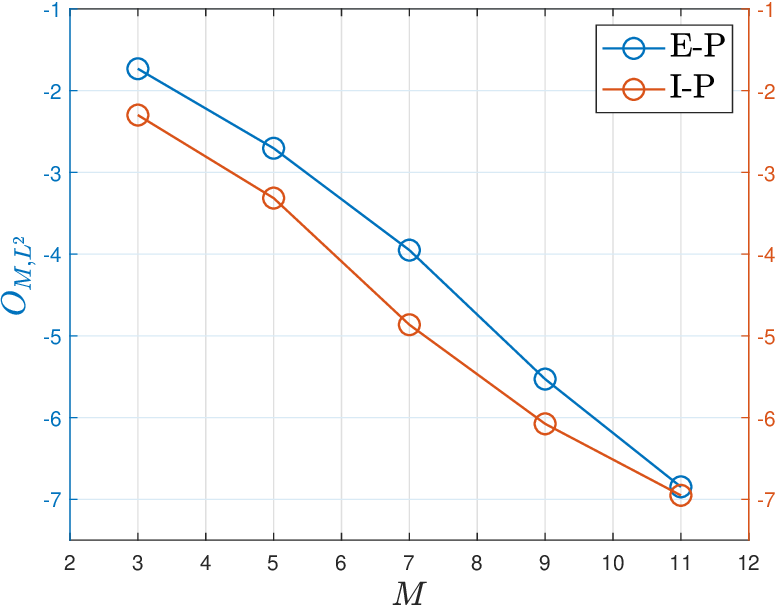}
		}
    \subfigure[even $M$]{
    \includegraphics[width=0.47\textwidth]{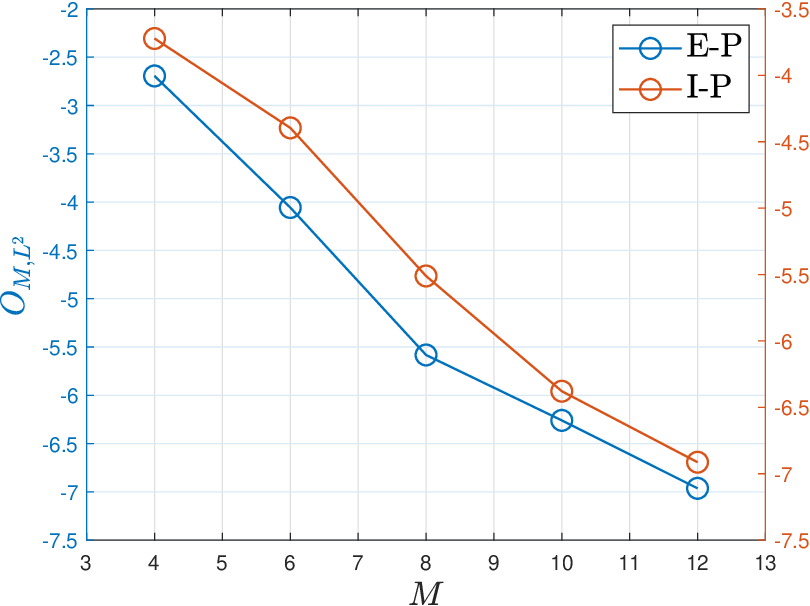}
		} 
    \caption{(Order of accuracy in Sec. \ref{sec:accuracy_two_pop}) The spectral convergence of the scheme \eqref{eq:two_fully_dis} in the spatial discretization for the two-population model. Here, the blue line is for the excitatory population (E-P) and the red line is for the inhibitory population (I-P). (a) $M$ is odd. (b) $M$ is even.} 
    \label{fig:ex2_spatial}
\end{figure}
	
Next, the spatial convergence of LLSGM \eqref{eq:two_fully_dis} is investigated by computing the $L^2$ error between numerical solutions obtained with different expansion numbers $M$ and the same reference solution. The $L^2$ error is calculated by \eqref{eq:ex1_error_spatial}. In this simulation, the time step length is fixed at $\Delta t = 0.001$, while the expansion number $M$ varies from $3$ to $12$. Fig. \ref{fig:ex2_spatial} displays the error obtained by \eqref{eq:ex1_error_spatial} with odd and even $M$ plotted separately, where a linear relationship between the error \eqref{eq:ex1_error_spatial} and $M$ is revealed, demonstrating the spectral convergence of LLSGM for the more complex two-population model \eqref{eq:two population}.

\subsubsection{Stability test for one-population model}
\label{sec:stability}
In this section, we assess the stability of LLSGM applied to the one-population model \eqref{eq:1}. The initial distribution \eqref{eq:initial} is set as $v_0=-1$ and $\sigma_0^2=0.5$, and the parameters in \eqref{eq:1} are $a=1, a_1=0.1$ and $b=0$. Here, the reference solution is obtained by the finite difference method. The $L^2$ error between the numerical solution and the reference solution at $t=0.2$ for different expansion numbers $M$ and time-step lengths $\Delta t$ is shown in Tab. \ref{tab:ex3_stability}. The $L^2$ error remains bounded and does not exhibit significant increases with changes in $M$ and $\Delta t$, demonstrating the numerical stability of LLSMG.

\begin{table}[!hptb]
	\centering 
	\def\arraystretch{1.5}
	\scalebox{0.8}{
		 \small
		\begin{tabular}{c|c|c|c|c|c|c}
			\diagbox{$M$}{$\Delta t$} & 0.1 & 0.05 & 0.025 & 0.0125 &0.0063 &0.0032 \\
			\hline
			3 & 1.63e-01 & 1.69e-01 & 1.73e-01 & 1.75e-01 & 1.76e-01 & 1.76e-01 \\
			4 & 6.96e-02 & 6.86e-02 & 6.81e-02 & 6.79e-02 & 6.77e-02 & 6.77e-02 \\
			5 & 6.27e-02 & 6.44e-02 & 6.55e-02 & 6.61e-02 & 6.65e-02 & 6.66e-02 \\
			6 & 2.04e-02 & 1.85e-02 & 1.78e-02 & 1.75e-02 & 1.74e-02 & 1.74e-02 \\
			7 & 2.01e-02 & 1.91e-02 & 1.90e-02 & 1.90e-02 & 1.91e-02 & 1.92e-02 \\
			8 & 1.11e-02 & 6.80e-03 & 4.81e-03 & 4.06e-03 & 3.81e-03 & 3.74e-03 \\
			9 & 1.09e-02 & 6.58e-03 & 4.69e-03 & 4.07e-03 & 3.92e-03 & 3.89e-03 \\
		   10 & 1.06e-02 & 5.91e-03 & 3.45e-03 & 2.36e-03 & 1.97e-03 & 1.85e-03 \\
		   11 & 1.05e-02 & 5.68e-03 & 3.03e-03 & 1.71e-03 & 1.15e-03 & 9.56e-04 \\
		   12 & 1.05e-02 & 5.68e-03 & 3.02e-03 & 1.67e-03 & 1.06e-03 & 8.38e-04 \\
		   13 & 1.05e-02 & 5.64e-03 & 2.94e-03 & 1.51e-03 & 7.76e-04 & 4.17e-04 \\
		\end{tabular}
	}
	\caption{(Stability test in Sec. \ref{sec:stability}) The $L^2$ error of LLSGM with different $M$ and $\Delta t$.}
    \label{tab:ex3_stability}
\end{table}

\subsection{Efficiency comparison}
\label{sec:efficiency}

In this section, the efficiency of LLSGM is validated by comparing it with the finite difference method (FDM) proposed in \cite{hu2021structure} and the Legendre-Galerkin method (LGM) in \cite{zhang2024spectral}. All simulations are performed on the model Intel Core i5-1135G7 CPU@2.40 GHz.
	
In the simulation, we consider the initial distribution \eqref{eq:initial} with $v_0=0$ and $\sigma_0^2=0.25$. The time-step length is fixed at $\Delta t=10^{-7}$ so that the time discretization error is negligible compared to the error in the spatial space. The final time is $t=0.5$ and parameters in \eqref{eq:1} are set as $a=1$ and $b=0.5$. Here, the reference solution $p_r$ for both LLSGM and LGM is the numerical solution with $M=30$, respectively. For FDM, the reference solution $p_r$ is the numerical solution obtained with spatial size $h=\frac{1}{128}$.
	
The $L^2$ error and the computational time for LLSGM, LGM, and FDM are presented in Tab. \ref{tab:ex3_time_compare_llgsm_lgm_fdm}. For the same expansion number, the error and the computational time of LLSGM are slightly less than those required by LGM in \cite{zhang2024spectral}. Furthermore, when the expansion number is small, the computational time of LLSGM is significantly less than that of FDM when the similar order of the numerical error is obtained. When the expansion number is larger, the computational time of LLSGM is much less than that of FDM. Specially, when the numerical error reaches $e^{-5}$, the computational time of FDM is more than ten times of LLSGM.


\begin{table}[htbp]
		\centering
          \renewcommand{\arraystretch}{1.5}
		\scalebox{0.8}{
        \small
			\begin{tabular}{c|c|c|c|c|c|c|c|c}
				\multicolumn{3}{c|}{LLSGM} & \multicolumn{3}{c|}{LGM} & \multicolumn{3}{c}{FDM} \\ \hline\hline
				$M$   &$\left \|p_{M}-p_r  \right \|_2$     &CPU time (s) &$M$ &   $\left \|p_{M}-p_r  \right \|_2$   &CPU time (s) &$h$ &  $\left \|p_{h}-p_r  \right \|_2$   &CPU time (s) \\ \hline
				4	      &3.55e-02            & 30.98  &4	      &6.18e-02       & 38.85  &1/4	       &3.75e-03       &47.72 \\
				8  	    &6.72e-03            & 47.55 &8       &9.51e-03      & 63.97 &1/8    	   &1.12e-03       &93.72 \\
				12      &1.33e-04          & 62.75 &12      &4.04e-04   	   & 73.48 &1/16         &3.12e-04   	&219.30  \\
				16      &2.11e-05            & 83.09 &16      &7.23e-05      & 96.21 &1/32   	   &8.16e-05   	&536.48 \\
				20      &1.96e-06  	     & 106.14 &20  &2.26e-06	  	& 139.73 &1/64   	   &1.98e-05   	&4741.07 \\
			\end{tabular}
		}
		\caption{(Efficiency test in Sec. \ref{sec:efficiency}) The computational time of LLSGM, LGM and FDM for \eqref{eq:1}.}
		\label{tab:ex3_time_compare_llgsm_lgm_fdm}
\end{table}

Next, considering the more complex two-population model \eqref{eq:two population}, the computational time for periodic solutions obtained by LLSGM and FDM is compared. In the simulation, the final time is set as $t=0.5$ and the time-step length is $\Delta t=10^{-7}$, while other parameters are listed in Tab. \ref{tab:ex2_periodic_para}. The corresponding computational time is provided in Tab. \ref{tab:ex3_compare_llgsm_fdm}. Significantly less time is required by LLSGM compared to FDM, demonstrating its greater efficiency for a more complex scenario.

\begin{table}[!hptb]
\centering
\renewcommand{\arraystretch}{1.5}
\scalebox{0.8}{
		 \small
    \begin{tabular}{c|>{\centering\arraybackslash}p{3cm} >{\centering\arraybackslash}p{2cm}}
        \multirow{1}{*}{\text{diffusion coefficient}}          & $a$                  & $1$ \\ \hline
	\multirow{1}{*}{\text{external input}}                 & $\nu_{E,\text{ext}}$ & $20$ \\ \hline
	\multirow{1}{*}{\text{refractory periods}}              & $\tau_E,\tau_I$      & $0.025$ \\ \hline
	\multirow{2}{*}{\text{initial distribution}}   & $[v_0^{E}, (\sigma_0^{E})^2]$              &  $[-1, 0.5]$ \\
	                                                & $[v_0^{I}, (\sigma_0^{I})^2]$              & $[-1, 0.5]$   \\ \hline
	\multirow{1}{*}{\text{synaptic delays}}        & $D_E^E, D_E^I, D_I^E, D_I^I$                     & $0.1$ \\ \hline
	\multirow{2}{*}{\text{connectivity parameters}}     & $[b_E^E,b_E^I]$                & $[3.5,4]$ \\  
	& $[b_I^E,b_I^I]$                & $[0.75,3]$ \\
    \end{tabular}}
\caption{(Efficiency test in Sec. \ref{sec:efficiency}) Computational parameters in \eqref{eq:two population} for the two-population model with delays and refractory states.}
\label{tab:ex2_periodic_para}
\end{table}

\begin{table}[!hptb]
		\centering 
		\def\arraystretch{1.5}
		\scalebox{0.8}{
        		 \small
			\begin{tabular}{c|c|c|c|c}
				\multirow{4}{*}{LLSGM} & $M$ & $\left \|p^E_{M}-p^E_r \right \|_2$ & $\left \|p^I_{M}-p^I_r \right \|_2$ & CPU time (s)  \\ \cline{2-5}
				& 4   & 3.97e-01  & 2.32e-01    & 65.72\\
				& 8   & 2.12e-02  & 2.27e-03    & 91.12\\
				& 12  & 1.47e-03  & 3.06e-04    & 122.29\\
				& 16  & 1.41e-04  & 2.78e-05    & 163.61\\
				\hline\hline
				\multirow{5}{*}{FDM}   &$h$ & $\left \|p_{h}^E-p_r^E  \right \|_2$ &$\left \|p_{h}^I-p_r^I  \right \|_2$  &CPU time (s)  \\ \cline{2-5}
				& 1/4  & 2.78e-02  & 1.97e-01    & 137.34\\
				& 1/8  & 8.02e-03  & 7.74e-02    & 213.06\\
				& 1/16 & 2.19e-03  & 1.94e-02    & 509.48\\
				& 1/32 & 7.78e-04  & 4.54e-03    & 1259.75\\
				& 1/64 & 3.31e-04  & 9.04e-04    & 11375.21\\
		\end{tabular}}
		\caption{(Efficiency test in Sec. \ref{sec:efficiency}) The computation time of LLSGM and FDM for \eqref{eq:two population}.}
		\label{tab:ex3_compare_llgsm_fdm}
\end{table}

\subsection{Blow-up events}
\label{sec:blow_up}
In this section, the blow-up events are first studied with LLSGM for the one-population model, and the occurrence of blow-up phenomena in the two-population model is then observed. The blow-up events have been widely studied in the NNLIF model, where the solution may blow up in a finite time.

\subsubsection{One-population model}
\label{sec:blow_up_one}
In this section, the blow-up events are studied in detail with LLSGM for the one-population model. The study in \cite{caceres2011analysis, zhang2024spectral} shows that the blow-up events occur if the parameter $b$ in \eqref{eq:1} is sufficiently large. In the simulation, the parameters are set as $a_0=1$, $a_1=0$ and $b=3$, and the scheme \eqref{eq:44} is utilized with the expansion number $M = 16$ and the time step length fixed as $\Delta t = 0.001$. 
	
The evolution of the mean firing rate $N(t)$ is plotted in Fig. \ref{fig:blow_up1_N}, and the distribution of the density function $p(v,t)$ on $v$ at $t =2.95, 3.15$ and $3.35$ is shown in Fig. \ref{fig:blow_up1_p}. It can be observed that the density function becomes increasingly concentrated around $V_R$, with a sharp increase in response to a rapid increase in the firing rate $N(t)$.

\begin{figure}[!hptb]
    \centering
    \subfigure[$N(t)$]{
        \includegraphics[width=0.45\textwidth]{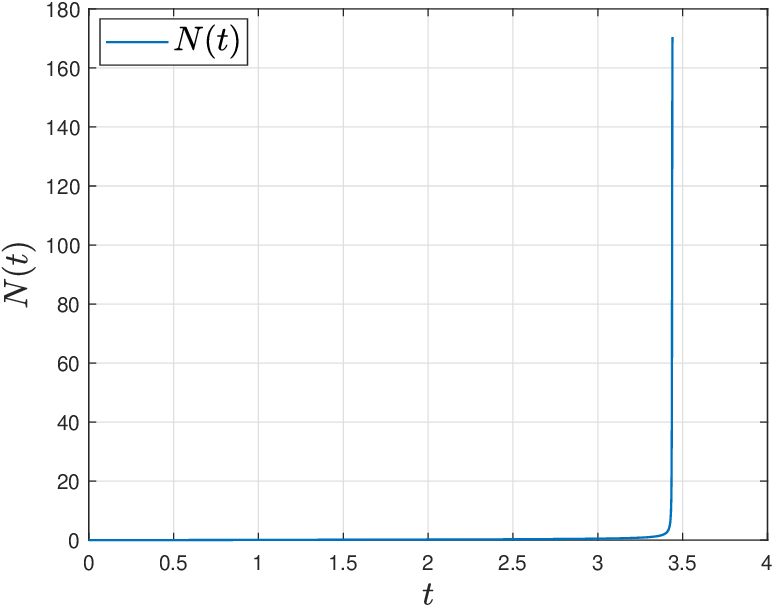}   
        \label{fig:blow_up1_N}
    }\hfill
    \subfigure[$p(v, t)$]{
        \includegraphics[width = 0.45\textwidth]{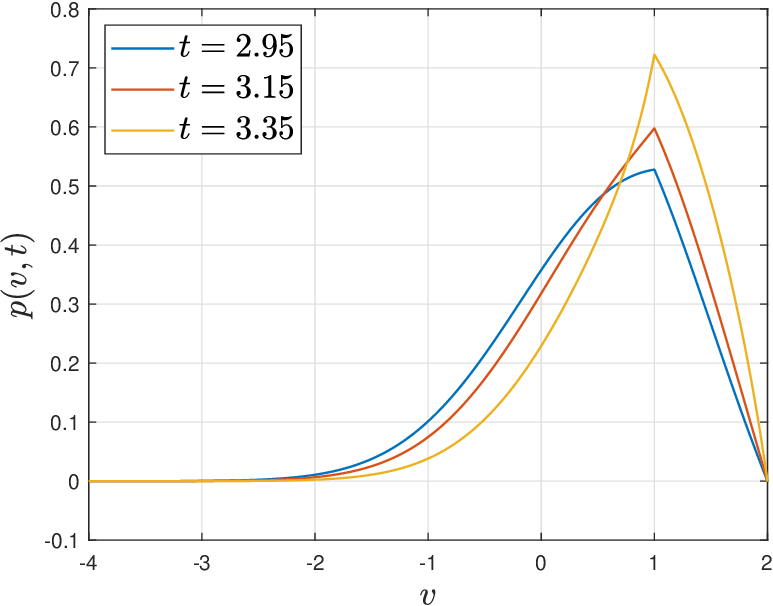}
        \label{fig:blow_up1_p}
    }
    \caption{(Blow-up events for the one-population model in Sec. \ref{sec:blow_up_one}) The blow-up phenomenon for the one-population model. (a) The evolution of the mean firing rate $N(t)$. (b) The density function $p(v,t)$ at $t = 2.95, 3.15$, and $3.35$.}
    \label{fig:blow_up1}
\end{figure}

\subsubsection{Two-population model}
\label{sec:blow_up_two}
In this section, we extend blow-up events by applying the LLSGM method to the excitatory-inhibitory population model \eqref{eq:two population}. As shown in \cite{caceres2016blow, Cáceres2018Analysis}, blow-up events can occur for the two-population model when the connectivity parameter $b_E^E$ becomes sufficiently large. In the simulation, we set $M_{\alpha}=N_{\alpha}$, and the scheme \eqref{eq:two_fully_dis} is adopted with the expansion number $M = 16$ and a fixed time-step length $\Delta t = 0.001$. The detailed parameters in \eqref{eq:two population} are listed in Tab. \ref{tab:ex2_blow_up_para}. 
	
\begin{table}[!hptb]
    \centering
    \renewcommand{\arraystretch}{1.5}
    \scalebox{0.8}{
    		 \small
    \begin{tabular}{c|>{\centering\arraybackslash}p{3cm} >{\centering\arraybackslash}p{2cm}}
		\multirow{1}{*}{\text{diffusion coefficient}}          & $a$                  & $1$ \\ \hline
		\multirow{1}{*}{\text{external input}}                 & $\nu_{E,\text{ext}}$ & $0$ \\ \hline
		\multirow{1}{*}{\text{refractory periods}}              & $\tau_E,\tau_I$      & $0$ \\ \hline
		\multirow{2}{*}{\text{initial distribution}}   & $[v_0^{E}, (\sigma_0^{E})^2]$              &  $[-1, 0.5]$ \\
		& $[v_0^{I}, (\sigma_0^{I})^2]$              & $[-1, 0.5]$   \\ \hline
		\multirow{1}{*}{\text{synaptic delays}}        & $D_E^E, D_E^I, D_I^E, D_I^I$                     & $0$ \\ \hline
		\multirow{2}{*}{\text{connectivity parameters}}     & $[b_E^E,b_E^I]$                & $[3,0.5]$ \\  
		& $[b_I^E,b_I^I]$                & $[0.75,0.25]$ \\
    \end{tabular}}
    \caption{(Blow-up events for the two-population model in Sec. \ref{sec:blow_up_two}) Computational parameters in \eqref{eq:two population} for the two-population model.}
    \label{tab:ex2_blow_up_para}
\end{table}

Fig. \ref{fig:blow_up2E_N} and Fig. \ref{fig:blow_up2I_N} illustrate the evolution of the mean firing rates $N_E(t)$ and $N_I(t)$, while Fig. \ref{fig:blow_up2E_p} and Fig. \ref{fig:blow_up2I_p} show the distribution of the density functions $p_E(v,t)$ and $p_I(v,t)$ at different time. As depicted in Fig. \ref{fig:blow_up2}, two populations nearly blow up simultaneously. The concentration of the density functions around $V_R$, as shown in Fig. \ref{fig:blow_up2E_p} and Fig. \ref{fig:blow_up2I_p}, provides a clear indication of the blow-up events. Compared to the inhibitory population, the density function of the excitatory population accumulates more distinctly at the reset voltage $V_R$, resulting in a sharper peak.

\begin{figure}[!hptb]
    \centering
    \subfigure[$N_E(t)$]{
        \includegraphics[width = 0.45\textwidth]{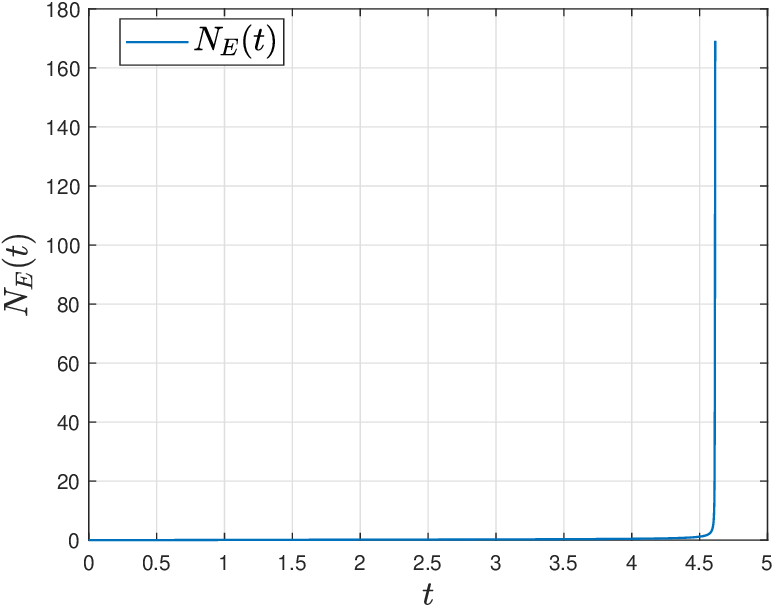}
        \label{fig:blow_up2E_N}
    }\hfill
    \subfigure[$p_E(v, t)$]{               
        \includegraphics[width = 0.45\textwidth]{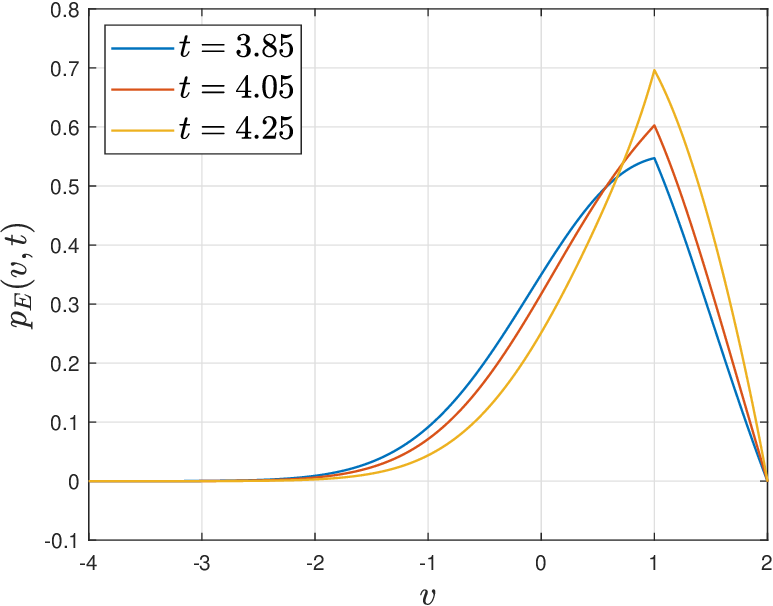}
        \label{fig:blow_up2E_p}
    }  \\
    \subfigure[$N_I(t)$]{    
        \includegraphics[width = 0.45\textwidth]{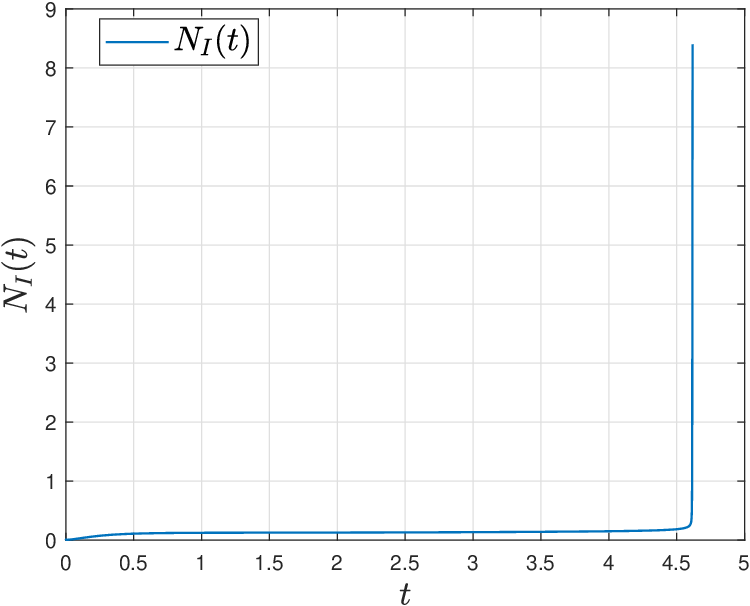}
        \label{fig:blow_up2I_N}
    }\hfill
    \subfigure[$p_I(v, t)$]{         
        \includegraphics[width = 0.47\textwidth]{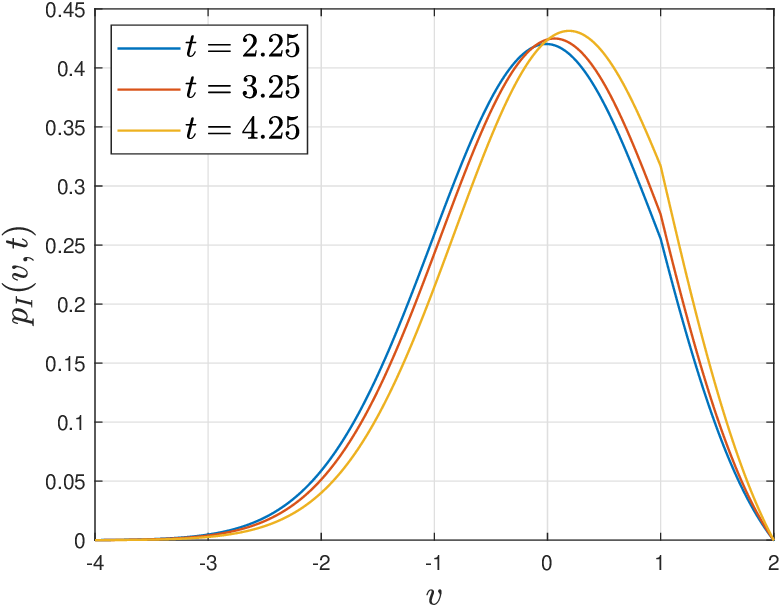}
        \label{fig:blow_up2I_p}
    }  
    \caption{(Blow-up events for the two-population model in Sec. \ref{sec:blow_up_two}) The blow-up phenomenon for the two-population model. The first row is the behavior of the mean firing rate $N_E(t)$, and the density function $p_E(v,t)$ at $t = 3.85, 4.05$, and $4.25$ for the excitatory population, and the second row is the behavior of the mean firing rate $N_I(t)$, and the density function $p_I(v,t)$ at $t = 2.25, 3.25$, and $4.25$ for the inhibitory population. }
    \label{fig:blow_up2}
\end{figure}

\subsection{Two-population model with delays and refractory states}
\label{sec:dely_refactory}

In this section, we perform numerical experiments to investigate the dynamics of the two-population model \eqref{eq:two population}, focusing on the significant roles of refractory states and synaptic delays. These factors are critical in shaping the collective behavior of neuron populations and are known to influence the emergence of periodic solutions, as shown in \cite{Cáceres2018Analysis}. However, the interactions between excitatory and inhibitory populations have been comparatively less explored. 
	\begin{figure}[!hptb]
		\centering
		\subfigure[$N_E(t), b_E^E=3.5$]{
			\includegraphics[width=0.45\textwidth]{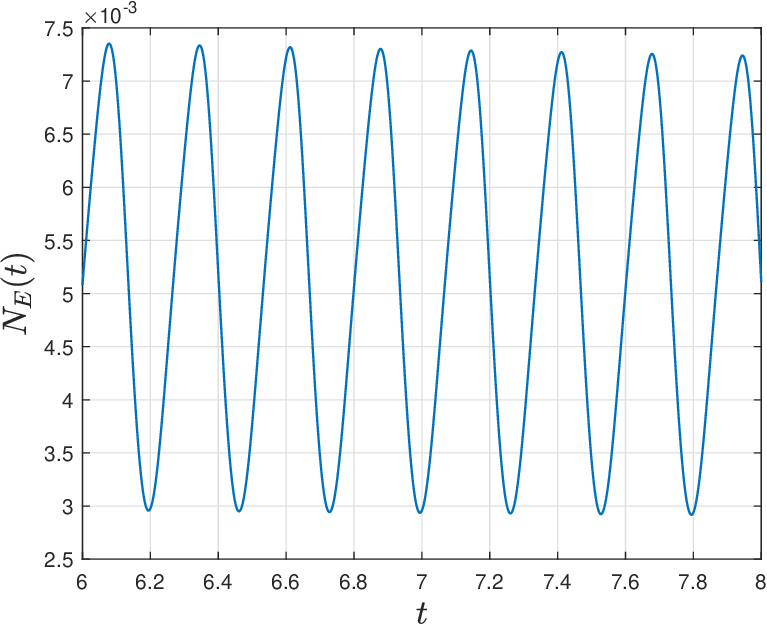}   
			\label{fig_periodic_NE}
		}\hfill
		\subfigure[$N_I(t), b_E^E=3.5$]{
			\includegraphics[width = 0.45\textwidth]{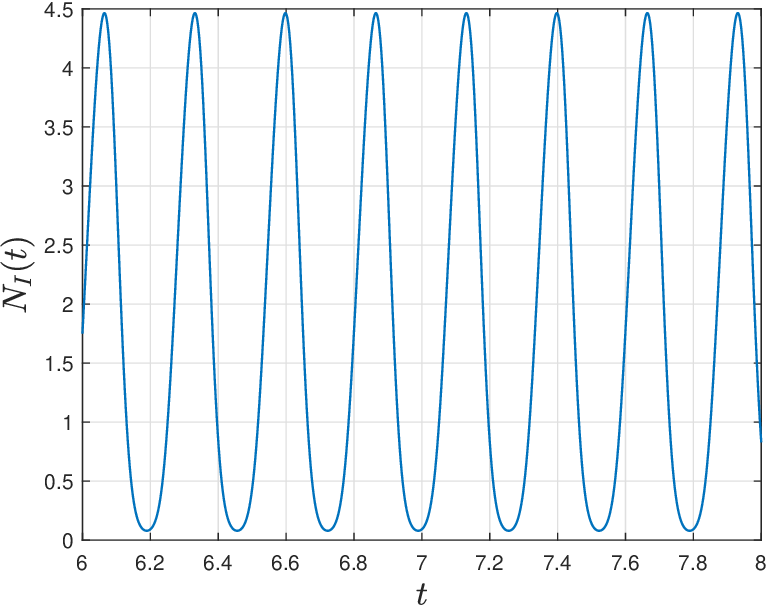}
			\label{fig_periodic_NI}
		}
		\caption{(Two-population model with delays and refractory states in Sec. \ref{sec:dely_refactory}) The evolution of the mean firing rates $N_{\alpha}(t), \alpha = E, I$ for $b_E^E = 3.5$, where the periodic oscillations are observed. }
	\end{figure}
To better understand the interaction between excitatory and inhibitory populations in the neuron networks, specifically, we consider a special scenario where an inhibitory population, exhibiting periodic oscillations under external stimuli due to synaptic delays and refractory states \cite{Cáceres2018Analysis}, is coupled with an excitatory population. By varying the connectivity between the two populations, we aim to observe the dynamic behavior of the system, including transitions between periodic oscillations, steady states, and blow-up events. 
	\begin{figure}[!hptb]
		\centering
		\subfigure[$N_E(t), b_E^E = 3.82$]{
			\includegraphics[width=0.455\textwidth]{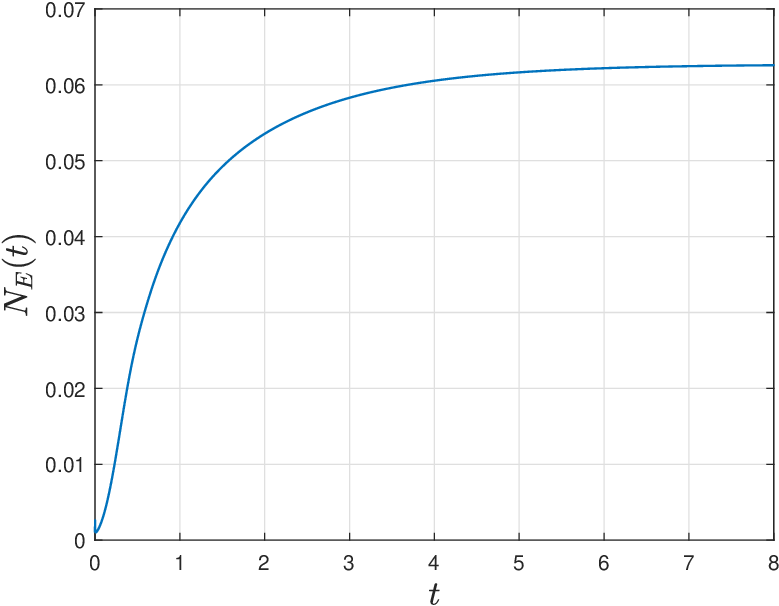}
			\label{fig_steady_NE}
		}\hfill
		\subfigure[$N_I(t), b_E^E = 3.82$]{
			\includegraphics[width=0.45\textwidth]{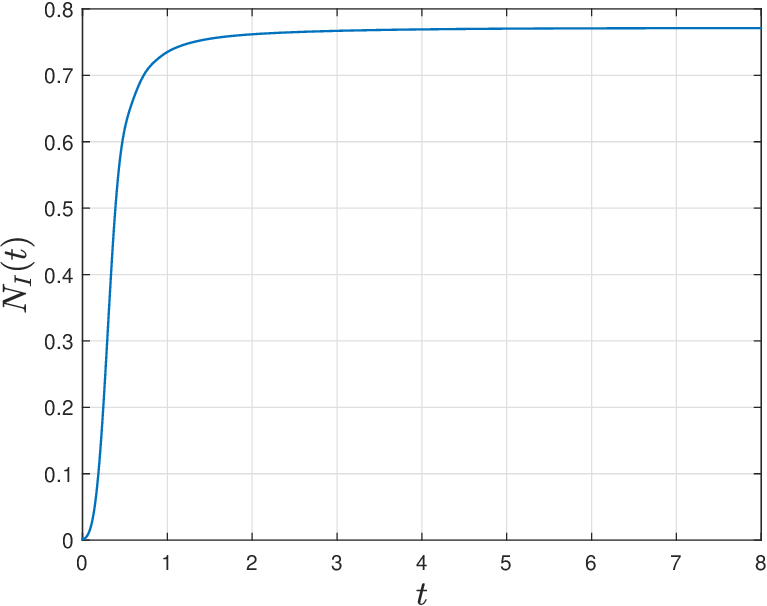}
			\label{fig_steady_NI}
		}
		\caption{(Two-population model with delays and refractory states in Sec. \ref{sec:dely_refactory}) The evolution of the mean firing rates $N_{\alpha}(t), \alpha = E, I$ for $b_E^E = 3.82$, where the phemomenon of steady state are observed.}
	\end{figure}
	
In the simulation, the expansion number is fixed at $M=16$, with a time-step length of $\Delta t=10^{-7}$, and the other parameters are listed in Tab. \ref{tab:ex2_periodic_para}. Since the blow-up events occur for large enough excitatory synaptic strength $b_E^E$ \cite{caceres2016blow}, we systematically vary $b_E^E$ while keeping other parameters fixed. For numerical convenience, the synaptic delays are considered integral multiples of the time-step length. When the excitatory synaptic strength $b_E^E$ is small, the dynamic behavior of the system is expected to resemble that of a single inhibitory population. As $b_E^E$ increases, the system is anticipated to transition to a steady state. For sufficiently large values of $b_E^E$, blow-up events are expected to occur in the system.

Results from the numerical experiments confirm the expected transitions. When the excitatory synaptic strength $b_E^E$ is relatively small, it seems that the inhibitory population dominates the system, and the firing rates $N_{\alpha}(t)$ exhibit periodic oscillations, as shown in Fig. \ref{fig_periodic_NE} and Fig. \ref{fig_periodic_NI} for $b_E^E=3.5$. 

\begin{figure}[!hptb]
		\centering
		\subfigure[$N_E(t), b_E^E = 4$]{
			\includegraphics[width=0.45\textwidth]{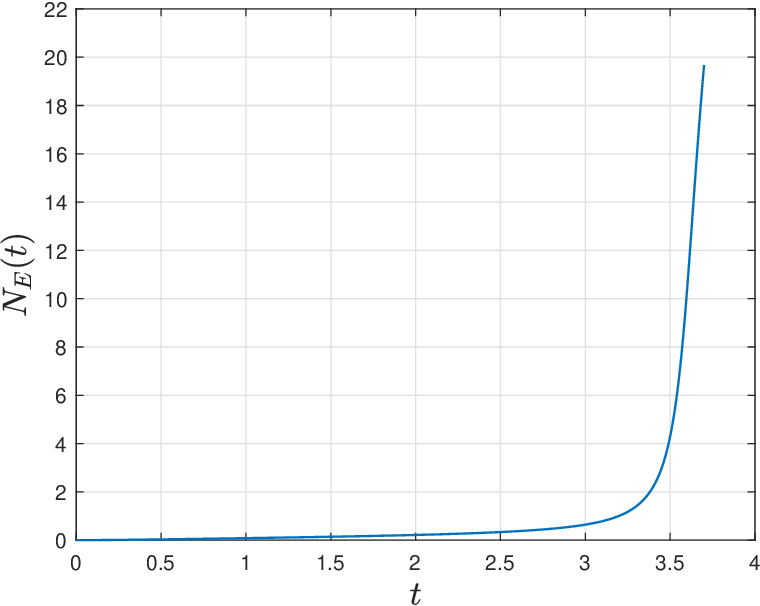}
			\label{fig_blowup_NE}
		}\hfill
		\subfigure[$N_I(t), b_E^E = 4$]{
			\includegraphics[width=0.45\textwidth]{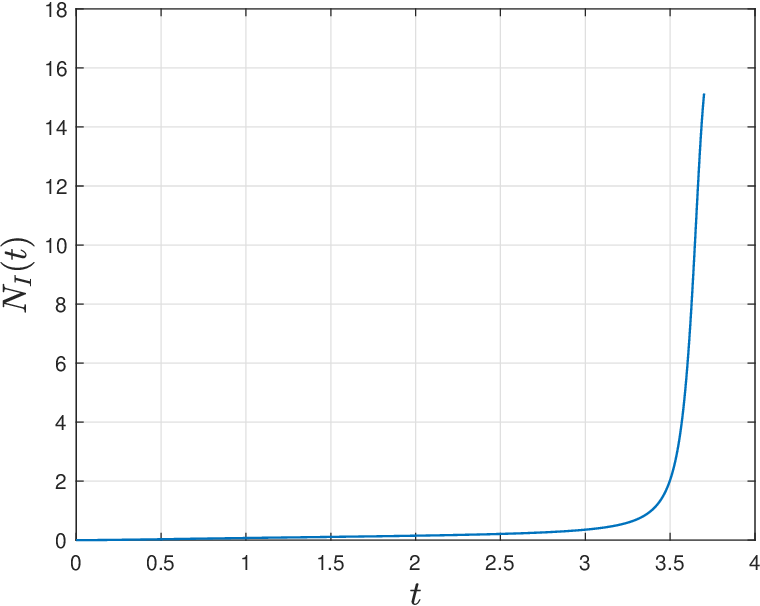}
			\label{fig_blowup_NI}
		}
		\caption{(Two-population model with delays and refractory states in Sec. \ref{sec:dely_refactory}) The evolution of the mean firing rates $N_{\alpha}(t), \alpha = E, I$ for $b_E^E = 4$, where the phenomenon of blow-up are observed.}
	\end{figure}

As $b_E^E$ increases, the influence of the excitatory population seems to intensify, gradually weakening the dominance of the inhibitory population, and shifting the system towards a steady state. In this case, both the excitatory firing rate $N_E(t)$ and inhibitory firing rate $N_I(t)$ increase without explosion and become constant after a certain period, depicted in Fig. \ref{fig_steady_NE} and Fig. \ref{fig_steady_NI} for $b_E^E=3.82$. This transition occurs because the excitatory synaptic strength is enhanced, leading to a more balanced interplay between the excitatory population and the inhibitory population.

Finally, when $b_E^E$ is sufficiently large, the excitatory population dominates the system, causing a blow-up event that leads to the overall blow-up in the system, as observed in Fig. \ref{fig_blowup_NE} and Fig. \ref{fig_blowup_NI} where at $b_E^E=4$. The firing rates $N_{\alpha}(t)$ grow without bound, consistent with the numerical experiments in \cite{Cáceres2018Analysis}.

\section{Conclusion} \label{sec:conclusion}

In this work, we have presented an efficient numerical scheme based on the spectral method to solve the Fokker-Planck equation in the semi-unbounded region associated with the Nonlinear Noisy Leaky Integrate-and-Fire model. The scheme is consistent with the Fokker-Planck equation and exhibits spectral convergence in the spatial direction. A comparison of computation time at the same accuracy level demonstrates the high efficiency of LLSGM. Additionally, we successfully extend the method to an excitatory-inhibitory population model with synaptic delays and refractory periods and explore the transitions between periodic solutions, steady states, and blow-up events. In the future, we may develop a more stable high-order numerical scheme in time and investigate more complex properties of the Fokker-Planck equation.

\section*{Acknowledgements}
We thank Prof. Jie Shen from Eastern Institute of Technology and Prof. Jiwei Zhang from Wuhan University for their valuable suggestions. The work of Zhennan Zhou is partially supported by the National Key R\&D Program of China (Project No. 2021YFA1001200, 2020YFA0712000), and the National Natural Science Foundation of China (Grant No. 12031013, 12171013). This work of Yanli Wang is partially supported by the President Foundation of China Academy of Engineering Physics (YZJJZQ2022017) and the National Natural Science Foundation of China (Grant No. 12171026, U2230402 and 12031013).

\bibliographystyle{plain}

\end{document}